\documentclass{article}
\usepackage[utf8]{inputenc}
\usepackage{amsthm}
\usepackage{graphicx}
\usepackage{amssymb}
\usepackage{color}
\usepackage[dvipsnames]{xcolor}
\usepackage{tikz}
\usepackage{tikz-cd}
\usetikzlibrary{shapes.geometric}
\usepackage{tcolorbox}
\usepackage{amsmath}
\usepackage[hidelinks]{hyperref}
\usepackage{float}
\usepackage{mathtools}
\usepackage{enumitem}
\usepackage{geometry}
\theoremstyle{plain}
\newtheorem{theorem}{Theorem}[section]
\newtheorem{lemma}[theorem]{Lemma}
\newtheorem{corollary}[theorem]{Corollary}

\newtheorem{proposition}[theorem]{Proposition}
\newtheorem{claim}[theorem]{Claim}
\newtheorem{problem}[theorem]{Problem}

\theoremstyle{definition}
\newtheorem{definition}[theorem]{Definition}

\theoremstyle{remark}
\newtheorem{remark}[theorem]{Remark}
\newtheorem*{remark*}{Remark}

\newcommand{\RR}{\mathbb{R}}

\newcommand{\ZZ}{{\mathbb Z}}

\newcommand{\TT}{\mathbb{T}}

\newcommand{\hcal}{\mathcal{H}}

\newcommand{\lcal}{\mathcal{L}}

\newcommand{\scal}{\mathcal{S}}

\newcommand{\ucal}{\mathcal{U}}
\newcommand{\zcal}{\mathcal{Z}}

\newcommand{\vol}{{\operatorname{Vol}}}

\def\Im{{\operatorname{Im}}}
\def\Re{{\operatorname{Re}}}

\def\Diff{{\operatorname{Diff}}}

\def\id{{\operatorname{id}}}

\def\path{\operatorname{Path}}

\usepackage[style=numeric, giveninits=true, sortcites=true, backend=biber]{biblatex} 
\AtEveryBibitem{\clearfield{isbn}}
\AtEveryBibitem{\clearfield{issn}}
\AtEveryBibitem{\clearfield{doi}}
\AtEveryBibitem{\clearfield{url}}
\AtEveryBibitem{\clearfield{number}}

\DeclareNameAlias{default}{given-family}

\DefineBibliographyExtras{english}{%
}

\renewbibmacro{in:}{
}

\DeclareFieldFormat*{title}{\mkbibemph{#1}}

\renewbibmacro*{title}{%
  \printfield{title}%
  \ifentrytype{misc}{\addcomma\addspace preprint}{}%
}

\DeclareFieldFormat{pages}{#1}

\DeclareFieldFormat*{journaltitle}{\textnormal{#1}}

\DeclareFieldFormat*{volume}{\mkbibbold{#1}}

\usepackage{minitoc}

\usepackage[toc,page]{appendix}

\title{Moduli of special Lagrangians with boundary, II: Lagrangian Flux and Affine Structures}
\author{Vasanth Pidaparthy\thanks{The author would like to thank Y. A. Rubinstein for suggesting this problem and for his constant encouragement and support. Research supported in part by NSF grants DMS-1906370, 2204347, and BSF grants 2016173, 2020329.}}
\date{\today}
\begin{document}

\maketitle

\begin{abstract}
    This article continues the study of moduli spaces of special Lagrangians with boundary in a Calabi--Yau manifold. The moduli space was shown to be a smooth finite-dimensional manifold in the prequel \cite{vasanth-hitchin-1}. This article investigates geometric structures on the moduli space of special Lagrangians with boundary and constructs a pair of special affine structures and a Hessian metric on this moduli space.
\end{abstract}

\section{Introduction}

This is the second of two articles that studies the moduli space of special Lagrangians with boundary (SLb) in a Calabi--Yau manifold $(X, \omega, J, \Omega)$, whose boundary is constrained along a given union of Lagrangian submanifolds (Definition~\ref{def:HWB-boundary-conditions-on-immersions-and-lagrangians}) \[\Lambda_1,\dots, \Lambda_d \subset X.\]
Building on our previous results on the unobstructedness of deformations \cite{vasanth-hitchin-1}, the goal of this article is to investigate affine structures on the moduli space of special Lagrangians with boundary.

In the prequel \cite{vasanth-hitchin-1}, we combined and generalized the works of McLean and Solomon--Yuval to observe that the infinitesimal generators of paths of any compact immersed special Lagrangian with boundary satisfying the boundary conditions of Solomon--Yuval (Definition~\ref{def:HWB-boundary-conditions-on-immersions-and-lagrangians}) are harmonic 1-forms vanishing on the boundary of the special Lagrangian. Consequently the moduli space of compact immersed special Lagrangians with boundary is a finite-dimensional manifold.

\begin{theorem}
\label{thm:HWB-tangentspacetospeciallagrangianswithboundary-2}
    {\rm\cite[Theorem~2.2]{vasanth-hitchin-1}} Under the assumptions of Proposition~\ref{prop:HWB-tangentspacetospeciallagrangianswithboundary-1}, the moduli space $\scal \lcal(X,L; \Lambda_1,\dots, \Lambda_d)$ is a finite-dimensional manifold whose tangent space at an immersed special Lagrangian is isomorphic to the space of closed and co-closed 1-forms on $L$ vanishing at the boundary. In particular $\scal \lcal$ is a manifold with $\dim \scal\lcal = \dim H^1(L,\partial L; \RR) = \dim H^{n-1}(L; \RR)$.
\end{theorem}

Following McLean, Hitchin studied properties of the finite-dimensional moduli space of closed special Lagrangian submanifolds of a Calabi--Yau \cite{hitchin1997moduli} and constructed a pair of special affine structures on the moduli space \cite[pp.~507--508]{hitchin1997moduli}. A special affine structure on a manifold is a family of charts whose transition maps are volume preserving affine maps. The affine structures were applied to study the Strominger--Yau--Zaslow (SYZ) conjecture, and to prove the SYZ conjecture on the standard flat complex torus, $X = \TT^n +\sqrt{-1} \TT^n$ \cite{SYZ-mirror-Tduality}.

\begin{problem}
\label{problem:HWB-3.2}
    Is the moduli space of special Lagrangians with boundary a special affine manifold?
\end{problem}

Hitchin also proved that the natural $L^2$ metric on the moduli space is a hessian metric. Special affine manifolds with Hessian metrics admit well defined real Monge--Amp\`ere operators. SYZ conjecture predicts that the geometry of the limiting affine manifold, in the large complex structure limit of a Calabi--Yau manifold, should be determined by a real Monge--Amp\`ere equation for the Hessian metric on the moduli space of special Lagrangian torii \cite{SYZ-mirror-Tduality}. The reader is referred to Hitchin \cite[Proposition 3]{hitchin1997moduli}, and surveys by Li \cite{YangLi_2022-syz-conjecture} and Hultgren \cite{hultgren2023dualityhessianmanifoldsoptimal-MR4853198}.

\begin{problem}
\label{problem:HWB-3.3}
    Is the $L^2$ metric on the moduli space of special Lagrangians with boundary a Hessian metric? 
\end{problem}

\section{Results}
\label{section:HWB-results}
The primary goal of this article is to resolve Problem~\ref{problem:HWB-3.2} and Problem~\ref{problem:HWB-3.3} for moduli spaces of compact immersed SLb satisfying the Lagrangian boundary conditions introduced by Solomon--Yuval (Definition~\ref{def:HWB-boundary-conditions-on-immersions-and-lagrangians}). Inspired by Solomon--Yuval \cite{SolomonYuval2020-geodesics}, the problems are studied in general for \textit{immersed} SLb. Our construction of Hitchin's affine structures proceeds via a pair of generalized Lagrangian flux functionals \eqref{eq:HWB-intro-rel-flux} and \eqref{eq:HWB-intro-dual-flux} \cite[{\S}5.3]{joyce-MR2167283}. The flux functionals are expected to inform future studies of moduli spaces of special Lagrangians.

This section defines the generalized Lagrangian flux functionals \eqref{eq:HWB-intro-rel-flux} and \eqref{eq:HWB-intro-dual-flux}. Hitchin's affine structures are constructed in Theorem~\ref{thm:HWB-affine-structure-on-moduli-space}. Hitchin's construction of the affine structures proceeded by studying period matrices on the tangent space to the moduli space \cite[pp.~507--508]{hitchin1997moduli}. While a similar approach can be extended to the case with boundary, we opt for a construction of the same affine structures in terms of the Lagrangian flux functional and its generalizations \eqref{eq:HWB-intro-rel-flux} and \eqref{eq:HWB-intro-dual-flux}. The flux functional gives Hitchin's affine structures a geometric interpretation and can be found in Joyce \cite[{\S}5.3]{joyce-MR2167283}.

Consider a smooth path of special Lagrangian immersions with boundary conditions $\Lambda_1,\dots, \Lambda_d$ (Definition~\ref{def:HWB-boundary-conditions-on-immersions-and-lagrangians}),
\[f: [0,1]\times L \rightarrow X,  \qquad f_t := f(t,\cdot). \]
The space of smooth special Lagrangian paths admits two natural Lagrangian flux functionals taking values respectively in $H^1(L, \partial L; \RR)$ and $H^{n-1}(L;\RR)$---the relative Lagrangian flux \eqref{eq:HWB-intro-rel-flux} and the special Lagrangian flux \eqref{eq:HWB-intro-dual-flux},

\begin{align}
    RF\left( f \right) :=&\ \left[ \int_0^1 f_t^* i_{\frac{df_t}{dt}}  \omega dt \right] \in H^1(L, \partial L; \RR), \label{eq:HWB-intro-rel-flux}\\
    SF\left( f \right) :=&\  \left[ \int_0^1 f_t^* i_{\frac{df_t}{dt}}  \Im(\Omega) dt\right] \in H^{n-1}(L;\RR) \label{eq:HWB-intro-dual-flux}.
\end{align}

The relative Lagrangian flux \eqref{eq:HWB-intro-rel-flux} is defined along any path of Lagrangians and is a closed form since $\omega$ vanishes identically along any such path (Lemma~\ref{lemma:HWB-formula-for-derivative-of-path-of-Lagrangians}). Similarly the special Lagrangian flux \eqref{eq:HWB-intro-dual-flux} is defined here for paths of special Lagrangian immersions, but is more generally well defined and a closed form along any path of immersions $f_t:L \rightarrow X$ satisfying $f_t^*\Im(\Omega) \equiv 0$ (Lemma~\ref{lemma:HWB-phit-is-closed}). A path of special Lagrangians enjoys having both functionals defined! Furthermore by Proposition~\ref{prop:HWB-tangentspacetospeciallagrangianswithboundary-1}, the integrands of $RF$ and $SF$ are Hodge duals to each other making $RF$ and $SF$ Poincar\'e--Lefschetz dual to each other.

Calabi first introduced the Lagrangian flux functional \eqref{eq:HWB-intro-rel-flux} for closed manifolds, associating an element of the first cohomology to each path of closed embedded Lagrangians \cite[(3.11)]{Calabi-grp-aut-sym-MR350776}, and has since played an important role in the study of moduli spaces of Lagrangians \cite{Fukaya-homological-Mirror-symmetry-MR1894935, Solomon2013TheCH}. Solomon--Yuval used the Lagrangian boundary conditions on paths of immersed Lagrangians (Definition \ref{def:HWB-boundary-conditions-on-immersions-and-lagrangians}) to generalized the Lagrangian flux functional to compact manifolds with boundary \cite[{\S}4.2]{SolomonYuval2020-geodesics}. They called the generalized functional the \textit{relative Lagrangian flux} since it took values in the relative first cohomology $H^1(L, \partial L; \RR)$. When $L = [0,1]\times N$ is a cylinder and $d=2$, then the relative Lagrangian flux recovers the Hamiltonian generating the geodesic of positive exact Lagrangians connecting $\Lambda_0$ and $\Lambda_1$, and dual to a component of the moduli space $\scal\lcal(X,L; \Lambda_0,\Lambda_1)$ via the cylindrical transform (Solomon--Yuval \cite[{\S}5]{SolomonYuval2020-geodesics}).

The \textit{special Lagrangian flux functional} $SF$ \eqref{eq:HWB-intro-dual-flux} is a special Lagrangian analog of the Lagrangian flux functional. A fundamental observation of Hitchin was to place the three real differential forms $\omega$, $\Re(\Omega)$ and $\Im(\Omega)$ on an equal footing in the studying a Calabi--Yau manifold \cite[{\S}2]{hitchin1997moduli}. A special Lagrangian is defined by the vanishing of two of these forms, namely $\omega$ and $\Im(\Omega)$ (Definition~\ref{def:immersedlagrangiansandspeciallagrangians}), so the special Lagrangian flux $SF$ is an analog of $RF$ and is expected to play an important role in the study of special Lagrangians.

In the notation of \eqref{eq:HWB-thm-1-theta-phi-notation}, $RF(f) = [\int_0^1 \theta_t dt]$ and $SF(f) = [\int_0^1 \phi_t dt]$. The key property of the flux functions is that they depend only on the end-point preserving homotopy class of the path $\zcal$. Fix a lifting $f_0:L \rightarrow X$ of a free immersed special Lagrangian $\zcal_0 \in \lcal(X,L; \Lambda_1,\dots, \Lambda_d)$ and a path $\zcal:[0,1] \rightarrow \lcal(X,L;\Lambda_1,\dots, \Lambda_d)$ starting at $\zcal_0$. Consider a smooth lifting $f:[0,1] \times L \rightarrow X$ of the path $\zcal$ (Definition \ref{def:smoothpathoflagrangians}) such that $f(0,\cdot) = f_0$.

\begin{theorem}
\label{thm:HWB-rel-flux-well-defined}
 $RF(f)$ is well defined by \eqref{eq:HWB-intro-rel-flux} and depends only on the lifting $f_0$ and the end point preserving homotopy class of $\zcal$ in $ \lcal(X,L; \Lambda_1, \dots, \Lambda_d)$.
\end{theorem}

\begin{theorem}
\label{thm:HWB-dual-flux-well-defined}
    $SF(f)$ is well defined by \eqref{eq:HWB-intro-dual-flux} and depends only on the lifting $f_0$ and the end point preserving homotopy class of $\zcal$ in $\scal \lcal(X,L; \Lambda_1,\dots, \Lambda_d)$.
\end{theorem}

\begin{remark}
    By Theorem~\ref{thm:HWB-rel-flux-well-defined} the relative and special Lagrangian flux of a path $\zcal:[0,1] \rightarrow \scal\lcal$ are independent of the chosen lifting and only depend on the homotopy class of the input path. Unfortunately they still depend on a lifting of one end point $\zcal_0$ of the path. This dependence arises because the flux functionals naturally take values in the cohomology of the immersed submanifold $f_0(L) = \zcal_0 \subset X$ and not the reference manifold $L$. A choice of immersion $f_0:L \rightarrow X$ representing $\zcal_0$ is necessary to identify the cohomology of $\zcal_0$ with that of $L$.
\end{remark}

Since $RF$ and $SF$ depend only on the end-point preserving homotopy class of the path, they restrict to maps on any simply-connected pointed neighborhood of the moduli space $\scal \lcal$ obtained by mapping any element to its flux along a path connecting it to the base-point.

For any simply-connected neighborhood $\ucal \subset \scal\lcal(X,L;\Lambda_1,\dots, \Lambda_d)$, a basepoint $\zcal_0 \in \ucal$ and a lifting $f_0$. Consider map (Definition~\ref{def:HWB-def-of-coordinate-chart})
\[R_{f_0}: \ucal \rightarrow H^1(L;\partial L; \RR) \quad \& \quad  S_{f_0}: \ucal \rightarrow H^{n-1}(L;\RR)\]
where $R_{f_0}(\zcal_1)$ (resp. $S_{f_0}$ is the relative (resp. special) Lagrangian flux of any lifting of a path joining $\zcal_0$ to $\zcal_1$ starting at the lifting $f_0$. It turns out that after shrinking $\ucal$ further, $R_{f_0}$ and $S_{f_0}$ are diffeomorphisms onto their image, and define the affine charts on the moduli space.

\begin{theorem}
\label{thm:HWB-affine-structure-on-moduli-space}
    For any point $\zcal_0 \in  \scal\lcal:= \scal\lcal(X,L; \Lambda_1,\dots, \Lambda_d)$ (Definition~\ref{def:HWB-moduli-space-of-special-lagrangians}), there exists a simply-connected open neighborhood $\ucal \subset \scal\lcal$ such that for any special Lagrangian immersion $f_0:L \rightarrow X$ representing $\zcal_0$, the maps $R_{f_0}$ and $S_{f_0}$ on $\ucal$ (Definition~\ref{def:HWB-def-of-coordinate-chart}) are diffeomorphisms onto their image.

    The moduli space $\scal\lcal$ admits an open cover by simply connected pointed coordinate charts of the form $(\ucal_{\alpha},\zcal_\alpha,f_\alpha)_{\alpha \in I}$ (Definition~\ref{def:HWB-pointed-coordinate-chart}) such that each of the two atlases
    \[(\ucal_{\alpha}, R_{f_\alpha}: {\alpha \in I}) \qquad \& \qquad (\ucal_{\alpha}, S_{f_\alpha}:{\alpha \in I}) \]
    have volume preserving affine transition maps.
\end{theorem}

Hitchin also observed that that product of the two affine charts $R_{f_0}$ and $S_{f_0}$ locally embeds the moduli space of closed special Lagrangians isometrically as a Lagrangian submanifold of $H^1(L,;\RR) \times H^{n-1}(L; \RR)$ \cite[pp.~508--510]{hitchin1997moduli}. This observation was used to prove that the $L^2$ metric on the moduli space is a Hessian metric. This observation generalizes to SLb.

\begin{proposition}
\label{prop:HWB-isometric-embedding-intro}
    Let $\ucal \subset \scal\lcal := \scal\lcal(X,L; \Lambda_1,\dot, \Lambda_d)$ be a simply connected neighborhood. Fix $\zcal_0 \in \ucal$ and a special Lagrangian immersion $f_0:L \rightarrow X$ representing $\zcal_0$. Recall that the product $H^{1}(L, \partial L; \RR) \times H^{n-1} (L; \RR)$ admits a natural indefinite metric $B$ and symplectic form $W$ (Definition~\ref{def:geometryofVtimesV*}). Then for
    \[F_{f_0}: = (R_{f_0}, S_{f_0}) : \ucal \rightarrow \scal\lcal, \]
    $F_{f_0}^*B$ coincides with an $L^2$ metric on $\scal\lcal$ (Definition~\ref{def:HWB-L2-iner-product}) and is a Hessian metric. $F_{f_0}^*W \equiv 0$ making $F_{f_0}$ a Lagrangian embedding.
\end{proposition}

Proposition~\ref{prop:HWB-isometric-embedding-intro} follows from Proposition~\ref{prop:HWB-isometric-embedding}. Hitchin's original argument generalizes to the setting of SLb \cite[pp.~508--510]{hitchin1997moduli}. We present a proof of the isometric and Lagrangian embedding using the Lagrangian flux functionals \eqref{eq:HWB-intro-rel-flux} and \eqref{eq:HWB-intro-dual-flux} as an application of these objects.



\textbf{Organization.} Section~\ref{sec:HWB-preliminaries} sets up the notation and definitions related the moduli of immersed Lagrangians and special Lagrangians following Akveld--Salamon \cite{Akveld2001} and Solomon--Yuval \cite{SolomonYuval2020-geodesics}. Section~\ref{sec:HWB-lagrangian-flux} is concerned with the relative Lagrangian flux \eqref{eq:HWB-intro-rel-flux} and special Lagrangian flux \eqref{eq:HWB-intro-dual-flux}. Theorems~\ref{thm:HWB-rel-flux-well-defined} and \ref{thm:HWB-dual-flux-well-defined} are proved in Section~\ref{sec:HWB-lagrangian-flux}. Section~\ref{sec:HWB-canonicalcoordinatesandaffinestructure} contains the construction of Hitchin's affine structures on the moduli space of SLb and a proof of Theorem~\ref{thm:HWB-affine-structure-on-moduli-space}. Section~\ref{subsec:HWB-isometric-and-lagrangian-embedding} discusses the isometric and Lagrangian local embedding of the moduli space of SLb (Proposition~\ref{prop:HWB-isometric-embedding-intro}). Finally Section~\ref{sec:HWB-generalizations} briefly discusses generalizations of the results of the article to almost Calabi--Yau manifolds.

\section{Preliminaries}
\label{sec:HWB-preliminaries}

Section~\ref{sec:modulispaceoflagrangianswithboundary} recalls the symplectic geometry of immersed manifolds with boundary Calabi--Yau manifold following the conventions in Solomon--Yuval \cite[{\S}2]{SolomonYuval2020-geodesics}. The Lagrangian boundary conditions defining the moduli space of immersed Lagrangians and special Lagrangians are recalled in Definitions \ref{def:HWB-boundary-conditions-on-immersions-and-lagrangians} and \ref{def:HWB-moduli-space-of-special-lagrangians}. Section~\ref{subsec:HWB-smooth-paths-and-tangents} recalls the notion of a smooth path of immersed SLb, and the associated tangent 1-form. These were introduced by Akveld--Salamon \cite[{\S}2]{Akveld2001}.

\subsection{Special Lagrangians and Lagrangian boundary conditions}
\label{sec:modulispaceoflagrangianswithboundary}

\begin{definition}
    \label{def:HWB-calabi-Yau-manifold}
    A quadruple $(X,\omega,J,\Omega)$ is called a \textit{Calabi--Yau manifold} if $(X,J)$ is a K\"ahler manifold equipped with a Ricci flat K\"ahler metric $\omega$ and a holomorphic top form $\Omega$ of constant length. Since $\Omega$ has constant length it satisfies
\begin{equation}
    (-1)^{\frac{n(n-1)}{2}} \left( \frac{\sqrt{-1}}{2}\right)^n \Omega \wedge \overline{\Omega} = \frac{1}{n!} \omega^n.
    \label{eq:HWB-calabi-Yau-equation}
\end{equation}
The Riemannian metric on $X$ is denoted by
\begin{equation}
    \label{eq:metriconakahlermanifold}
    g := \omega(\cdot, J \cdot).
\end{equation}
\end{definition}

In this article a Calabi--Yau manifold is assumed to be Ricci flat. A K\"ahler manifold with trivial canonical bundle which is not Ricci flat is called \textit{almost Calabi--Yau}. Special Lagrangians are not minimal submanifolds in an almost Calabi--Yau manifold, but instead are minimal submanifolds for a Riemannian metric conformal to \eqref{eq:metriconakahlermanifold}. The results of this article directly generalize to almost Calabi--Yau manifolds (see Section~\ref{sec:HWB-generalizations}.

\begin{definition}~
\label{def:immersionsandimmersedforms}
\begin{enumerate}[label={(\arabic*)}]
    \item Let $L$ be a smooth compact oriented manifold of dimension $n$ with smooth boundary. $\Diff(L)$ denotes the group of diffeomorphisms of $L$ that preserve boundary components of $L$, i.e., $\psi \in \Diff(L)$ if
    \begin{enumerate}[label={(\roman*)}]
        \item $\psi:L \rightarrow L$ is a diffeomorphism and
        \item $\psi(C) \subset C$ for every connected component of the boundary $C\subset \partial L$.
    \end{enumerate}

    \item An \textit{immersion} $f:L \rightarrow X$ is a smooth map for which the derivative $df_p: T_pL \rightarrow T_{f(p)}X$ is injective for every $p\in L$. An immersion $f$ is called \textit{free} if $f\circ \psi = f$ and $\psi \in \Diff(L)$ implies $\psi = \id_L$, i.e., $f$ has trivial stabilizer under the reparametrization action of $\Diff(L)$. An injective immersion is always free, and when $L$ is compact an injective immersion is also an embedding.
    
    \item An \textit{immersed submanifold of $X$ of type $L$} is an equivalence class of immersions $f:L \rightarrow X$ under the reparametrization action of $\Diff(L)$,
    \[f \sim f' \quad \iff \quad  \exists \ \psi\in \Diff(L)\quad s.t. \quad f' = f\circ \psi. \]
    The equivalence class of an immersion $f$ is denoted by $[f]$.
    
    \item An immersed submanifold is called \textit{free} if it admits a representative that is a free immersion. It is immediate that every representative is free.
\end{enumerate}
\end{definition}

The results of this article are formulated in terms of immersed special Lagrangians. However most calculations proceed by fixing a specific parametrization or representative immersion, and the distinction between immersions and embeddings does not play a serious role in this article.

\begin{definition}~
\label{def:immersedlagrangiansandspeciallagrangians}
\begin{enumerate}[label={(\arabic*)}]
    \item An immersion $f:L \rightarrow X$ is \textit{Lagrangian} if $f^*\omega = 0$. Note that being Lagrangian is $\Diff(L)$ invariant.

    \item A Lagrangian immersion $f:L \rightarrow X$ is called \textit{special Lagrangian} if in addition $f^*\Im(\Omega) = 0$.
    
    \item An immersed manifold $[f]$ of type $L$ in $X$ is (special) Lagrangian if one representative is (special) Lagrangian. Note that if one representative is (special) Lagrandian then all representatives are so.
\end{enumerate}
\end{definition}

Definition \ref{def:HWB-boundary-conditions-on-immersions-and-lagrangians} recalls the Lagrangian boundary conditions for Lagrangians with boundary, introduced by Solomon--Yuval \cite[Notation 2.7]{SolomonYuval2020-geodesics}.

\begin{definition}
\label{def:HWB-boundary-conditions-on-immersions-and-lagrangians}
Let $L$ be a compact manifold with smooth boundary, and let $C_1,\dots, C_d$ denote its boundary components. Fix an ordering $C_1,\dots, C_d$ on the boundary components. Fix pairwise disjoint embedded Lagrangian submanifolds $\Lambda_1, \dots, \Lambda_d \subset X$.
\begin{enumerate}
    \item A smooth map $f:L \rightarrow X$ is said to have boundary condition $\Lambda_1,\dots, \Lambda_d$ if
\begin{enumerate}[label={(\alph*)}]
    \item $f(C_i) \subset \Lambda_i$ for $i=1,\dots, d$, and
    \item $df(T_pL) \not\subset T_{f(p)} \Lambda_i$ for all $p\in C_i$ and $i=1,\dots, d$.
\end{enumerate}
    The set of smooth maps with boundary conditions $\Lambda_1,\dots, \Lambda_d$ is denoted by $C^\infty(L, X; \Lambda_1,\dots, \Lambda_d)$.

    If $f$ is in addition a free immersion, then $f|_{C_i}: C_i \rightarrow \Lambda_i$ is an immersion for $i=1,\dots, d$. Note that both (a) and (b) are invariant under reparametrization by $\Diff(L)$ because by Definition \ref{def:immersionsandimmersedforms}(1) elements of $\Diff(L)$ preserve boundary components.

\item An immersed submanifold of type $L$ is said to have boundary conditions $\Lambda_1, \dots, \Lambda_d$ if one and hence every representative immersion has these boundary condition. 
\end{enumerate}
Implicit in Definition~\ref{def:HWB-boundary-conditions-on-immersions-and-lagrangians}.1 is an ordering of the $d$ boundary components of $L$ corresponding to the ordering $\Lambda_1,\dots, \Lambda_d$.
\end{definition}

\begin{remark}
    There have been other notions of boundary conditions in the study of special Lagrangians. Butscher studied a different class of special Lagrangians whose boundary is constrained to lie on a fixed co-dimension 2 symplectic submanifold \cite[p.~3]{Butscher2002}, and generalized McLean's deformation theorem under these boundary conditions \cite{Mclean1998}. The resulting moduli space of special Lagrangians is fairly different from the one studied in this article, and its tangent space is identified with harmonic 1-forms satisfying a Neumann boundary condition (compare with Theorem~\ref{thm:HWB-tangentspacetospeciallagrangianswithboundary-2}).
\end{remark}
 
\begin{definition}
\label{def:HWB-moduli-space-of-special-lagrangians}
Let $L$ have $d$ boundary components $C_1,\dots, C_d$, and let $\Lambda_1, \dots, \Lambda_d \subset X$ be pairwise disjoint embedded Lagrangians.
\begin{enumerate}[label={(\arabic*)}]
    \item The moduli space of smooth free immersed Lagrangian submanifolds of $X$ of type $L$ with boundary conditions $\Lambda_1,\dots, \Lambda_d$ is denoted by
    \[\lcal := \lcal(X,L; \Lambda) := \lcal (X,L ; \Lambda_1,\dots, \Lambda_d).\]
    \item The moduli space of smooth free immersed special Lagrangian submanifolds in $X$ of type $L$ with boundary conditions $\Lambda_1,\dots, \Lambda_d$  is denoted by
    \begin{gather*}
        \scal\lcal :=  \scal \lcal(X,L; \Lambda) :=\scal \lcal (X,L ; \Lambda_1,\dots, \Lambda_d).
    \end{gather*}
\end{enumerate}
As in Definition~\ref{def:HWB-boundary-conditions-on-immersions-and-lagrangians}, the boundary components $C_1,\dots, C_d$ of $L$ and the Lagrangians $\Lambda_1,\dots, \Lambda_d$ are assumed to be ordered. The boundary data $\Lambda_1,\dots, \Lambda_d$ is often abbreviated to $\Lambda$ or omitted to simplify notation.
\end{definition}

  \subsection{Smooth paths and tangents in the moduli space of Lagrangians with boundary}
  \label{subsec:HWB-smooth-paths-and-tangents}

  This section recalls the notion of a smooth map of Lagrangian immersion introduced by Akveld--Salamon's \cite[p.~613]{Akveld2001}, and the tangent vector to such a smooth path \cite[Lemma~2.1]{Akveld2001}.

    \begin{definition}
    \label{def:smoothpathoflagrangians}
    Let $Y$ be a simply connected smooth manifold, and consider a map $\zcal:Y \rightarrow \lcal(X,L; \Lambda_1,\dots, \Lambda_d)$ (Definition~\ref{def:HWB-moduli-space-of-special-lagrangians}). A smooth (resp. $C^{k,\alpha}$) lifting of $\zcal$ is a smooth (resp. $C^{k,\alpha}$) function \[f:Y \times L \rightarrow X, \qquad f_y := f(y,\cdot)\]
    such that $f_y:L \rightarrow X$ is a free Lagrangian immersion representing the immersed Lagrangians submanifold $\zcal(y)$ for all $y\in Y$. The map $\zcal$ is called \textit{smooth} (resp. $C^{k,\alpha}$) if it admits a \textit{smooth (resp. $C^{k,\alpha}$) lifting}.
    \end{definition}

    Theorem~\ref{thm:HWB-tangentspacetospeciallagrangianswithboundary-2} asserts that the moduli space $\scal\lcal (X,L; \Lambda_1,\dots, \Lambda_d)$ is a smooth manifold. But Definition \ref{def:smoothpathoflagrangians} gives a different definition of smooth maps in $\scal\lcal$. It turns out that both definitions are compatible by Corollary~\ref{corollary:HWB-lifted-local-coordinate-chart} which constructs a local coordinate chart on the smooth moduli space from Theorem~\ref{thm:HWB-tangentspacetospeciallagrangianswithboundary-2}, which is smooth in the sense of \eqref{def:smoothpathoflagrangians}.
    
     \begin{corollary}
         \label{corollary:HWB-lifted-local-coordinate-chart}
         Fix a free immersed special Lagrangian $\zcal_0 \in \scal\lcal(X,L; \Lambda_1,\dots, \Lambda_d)$ and a free special Lagrangian immersion $f_0:L \rightarrow X$ representing it. Also fix a basis
         \[\theta_1,\dots \theta_m  \in \ker d_0 F = Z^1_D(L) \cap cZ^1(L)\]
         of harmonic fields with Dirichlet boundary conditions on $L$, with respect to the pullback metric $f_0^*g$.  Let $t_1,\dots, t_m$ denote coordinates on $\RR^m$. There also a simply-connected domain $0\in U \subset \RR^m$ and a smooth map
         \[f: U\times L \rightarrow X \]
         satisfying the following:
         \begin{enumerate}[label={(\roman*)}]
             \item $f(t,\cdot):L \rightarrow X$ is a free special Lagrangian immersion with boundary in $\Lambda_1,\dots, \Lambda_d$ for all $t \in U$,
             \item $f(0,\cdot) = f_0 $, the given immersion representing $\zcal_0$, and
             \item  $\displaystyle{f_0^* \left( \left. i_{\partial_{i} f}  \omega \right|_{t=0} \right) = \theta_i} $ for $i =1,\dots, m$, where $\partial_i f := \frac{\partial f}{ \partial t_i}$.
         \end{enumerate}
         Consequently the map $t \mapsto [f_t]$ defines a smooth diffeomorphism (Definition~\ref{def:smoothpathoflagrangians}) from $U\subset \RR^m$ to some neighborhood $\ucal \subset \scal\lcal (X,L; \Lambda_1, \dots, \Lambda_d)$ of $\zcal_0$ mapping $0\in U$ to $\zcal_0 \in \ucal$.
     \end{corollary}

     \begin{proof}
         The reader is referred to the prequel for a proof \cite[Corollary~4.9]{vasanth-hitchin-1}.
     \end{proof}

    Akveld--Salamon identified the tangent vector to a smooth path of Lagrangian immersions $f_t:L \rightarrow X$, $t\in (-\epsilon, \epsilon)$, with the 1-form \eqref{eq:HWB-formula-for-derivative-of-path-of-Lagrangians} \cite[Lemma~2.1]{Akveld2001}. The tangent 1-form features in Proposition~\ref{prop:HWB-tangentspacetospeciallagrangianswithboundary-1} and is also in the definition of the Relative Lagrangian flux \eqref{eq:HWB-intro-rel-flux}.

  \begin{lemma}
  \label{lemma:HWB-formula-for-derivative-of-path-of-Lagrangians}
    Consider a smooth map $f:(-\epsilon , \epsilon) \times L \rightarrow X$ such that $f_t:=f(t,\cdot) $ is a Lagrangian immersion with boundary conditions $\Lambda_1,\dots, \Lambda_d$ for each $t\in (-\epsilon, \epsilon)$ (Definition~\ref{def:HWB-boundary-conditions-on-immersions-and-lagrangians}). Then
      \begin{equation}
      \label{eq:HWB-formula-for-derivative-of-path-of-Lagrangians}
          \theta_t :=  f_t^* \left(  i_{ \frac{df_t}{dt}} \omega \right)
      \end{equation}
      is a well defined closed 1-form satisfying $\theta_t|_{\partial L} \equiv 0$.
  \end{lemma}
  
   \begin{proof}~
       By Cartan's formula for the Lie derivative,
       \begin{equation}
           \label{eq:HWB-formula-for-derivative-of-path-of-Lagrangians-eq1}
           \theta = df_t^* i_{\frac{df_t}{dt}} \omega = { \frac{d}{dt} } (f_t^* \omega) - f_t^*i_{ \frac{df_t}{dt} } d\omega. 
       \end{equation}
       Since $\omega$ is closed, $d\omega = 0$. Since $f_t:L \rightarrow X$ is a Lagrangian immersion, $f_t^*\omega = 0$. Thus the RHS of \eqref{eq:HWB-formula-for-derivative-of-path-of-Lagrangians-eq1} vanishes proving that $\theta_t$ is closed. 
       
       Consider a boundary point $p\in \partial L$ and assume WLOG that $p\in C_i$, where $C_i$ is the boundary component of $L$ corresponding to the boundary Lagrangian $\Lambda_i \subset X$. Fix a tangent vector $\xi_p \in T_p\partial L = T_pC_i$. Then
       \begin{equation}
           \label{eq:HWB-formula-for-derivative-of-path-of-Lagrangians-eq2}
           \theta_t (\xi_p) = f_t^*\left(i_{ \frac{df_t}{dt} } \omega \right) (\xi_p) = \omega \left( \frac{df_t(p)}{dt} , df_t(\xi_p) \right).
       \end{equation}
       By Definition \ref{def:HWB-boundary-conditions-on-immersions-and-lagrangians}, $f_t(C_i) \subset \Lambda_i$ for all $t\in (-\epsilon ,\epsilon)$, and in particular $f_t(p) \in \Lambda_i$ for all $t\in (-\epsilon, \epsilon)$. Then
       \begin{equation}
           \label{eq:HWB-formula-for-derivative-of-path-of-Lagrangians-eq3}
           \frac{df_t(p)}{dt}  \in T_{f_t(p)}\Lambda_i \quad \&\quad df_t(\xi_p) \in T_{f_t(p)} \Lambda_i.
       \end{equation}
       But $\Lambda_i \subset X$ is Lagrangian by definition and so $\Omega|_{\Lambda_i} \equiv 0$. Combining this with \eqref{eq:HWB-formula-for-derivative-of-path-of-Lagrangians-eq2} and \eqref{eq:HWB-formula-for-derivative-of-path-of-Lagrangians-eq3}, conclude that
       \begin{equation}
           \theta_t(\xi_p) = 0 \quad \forall\ p\in \partial L, \xi_p \in T_pL, \quad \implies \quad \theta_t|_{\partial L} \equiv 0. \tag*{qedhere}
       \end{equation}
   \end{proof}

As discussed in \S\ref{section:HWB-results}, $\omega$ and $\Im(\Omega)$ are expected to play dual roles in the theory by an observation of Hitchin \cite[{\S}2]{hitchin1997moduli}. Analogous to the Akveld--Salamon tangent 1-form \eqref{eq:HWB-formula-for-derivative-of-path-of-Lagrangians} is the dual $(n-1)$-form obtained by swapping $\Im(\Omega)$ for $\omega$.

\begin{lemma}
\label{lemma:HWB-phit-is-closed}
Consider a smooth map $f:(-\epsilon , \epsilon) \times L \rightarrow X$ such that $f_t:=f(t,\cdot) \in C^\infty(L, X; \Lambda_1,\dots, \Lambda_d)$ (Definition~\ref{def:HWB-boundary-conditions-on-immersions-and-lagrangians}), and $f_t^* \Im(\Omega) \equiv 0$ for each $t\in (-\epsilon, \epsilon)$. Define
\begin{equation}
    \label{eq:HWB-formula-for-hodge-dual-of-derivative-of-path-of-Lagrangians}
    \phi_t := f_t^* \left( i_{\frac{df_t}{dt}} \Im(\Omega) \right) \qquad t\in (-\epsilon, \epsilon). 
\end{equation}
Then $\phi_t$ is a closed $(n-1)$-form.
\end{lemma}
\begin{proof}
By Cartan's formula for the Lie derivative,
\[\frac{d}{dt} f_t^*\Im(\Omega) = f_t^*\left(i_{\frac{df_t}{dt}} d \ \Im(\Omega) + d i_{\frac{df_t}{dt}}\ \Im(\Omega) \right),\] \begin{align*}
    \implies \qquad d\phi_t = d\, f_t^*\left(i_{\frac{df_t}{dt}} \Im(\Omega) \right) =&\ f_t^*\left(d\, i_{\frac{df_t}{dt}} \Im(\Omega) \right) \\
    =&\ \frac{d}{dt} f_t^*\Im(\Omega) - f_t^*\left(i_{\frac{df_t}{dt}} d \ \Im(\Omega)  \right)
\end{align*}
But $d \Im(\Omega)=0$ since $\Omega$ is holomorphic and hence closed, and $f_t^*\Im(\Omega) = 0$ for all $t\in[0,1]$ since $f_t$ is a special Lagrangian immersion. Thus $d\phi_t = 0$.
\end{proof}

It was proved in the prequel that the closed forms \eqref{eq:HWB-formula-for-derivative-of-path-of-Lagrangians} and \eqref{eq:HWB-formula-for-hodge-dual-of-derivative-of-path-of-Lagrangians} are Hodge dual to each other.

\begin{lemma}
\label{prop:HWB-tangentspacetospeciallagrangianswithboundary-1}
{\rm\cite[Lemma~4.2]{vasanth-hitchin-1}} Consider a smooth path of special Lagrangian immersions with Lagrangian boundary conditions (Definition~\ref{def:HWB-boundary-conditions-on-immersions-and-lagrangians})
\[f: [0,1] \times L \rightarrow X, \quad f\in C^\infty,\quad f_t:=f(t,\cdot).\]
Let $\star_t$ denote the Hodge star operator corresponding to the pullback Riemannian metric $g_t = f^*g$ on $L$, and
\begin{equation}
\label{eq:HWB-thm-1-theta-phi-notation}
    \begin{aligned}
    \theta_t:=&\ f_t^* \left( i_{\frac{df_t}{dt}}\omega \right) \\
    \phi_t :=&\ f_t^* \left( i_{\frac{df_t}{dt}} \Im(\Omega) \right) .
\end{aligned} \qquad t\in [0,1],
\end{equation}
Then $\star_t\theta_t = \phi_t$. In particular by Lemmas~\ref{lemma:HWB-formula-for-derivative-of-path-of-Lagrangians} and \ref{lemma:HWB-phit-is-closed}, $\theta_t$ and $\phi_t$ are closed and co-closed with respect to the pullback metric $f_t^* g$, and are harmonic.
\end{lemma}

 \subsection{\texorpdfstring{$L^2$ metric on the moduli space}{L\^2 metric on the moduli space}}

The moduli space $\scal\lcal := \scal\lcal(X,L, \Lambda_1,\dots, \Lambda_d)$ admits a natural $L^2$ inner product.

\begin{definition}
\label{def:HWB-L2-iner-product}
Fix a point $\zcal_0 \in \scal\lcal$ and a special Lagrangian immersion $f_0:L \rightarrow X$ representing $\zcal_0$. By Corollary~\ref{corollary:HWB-lifted-local-coordinate-chart}, every smooth path $\zcal:(-\epsilon, \epsilon) \rightarrow \scal\lcal$ admits smooth lifting $f:(-\epsilon, \epsilon) \times L \rightarrow X$, such that $f(0,\cdot) = f_0$ is the fixed immersion, and $f_t:= f(t,\cdot)$ is a special Lagrangian immersions with boundary conditions $\Lambda_1,\dots, \Lambda_d$ representing $\zcal_t$ for all $t\in (-\epsilon, \epsilon)$.
\begin{enumerate}
    \item Fix a smooth path $\zcal:(-\epsilon, \epsilon) \rightarrow \scal\lcal$. The \textit{tangent 1-form to the path $\zcal_t$ at $\zcal_0$ with respect to lifting $f_0:L \rightarrow X$ of $\zcal_0$} is the 1-form
    \[\theta_{\zcal, f_0} := f_0^* \left( i_{\left. \frac{df_t}{dt} \right|_{t=0}} \omega \right).\]
    Note that $\theta_{\zcal, f_0}$ is a closed 1-form vanishing on $\partial L$ by Lemma~\ref{lemma:HWB-formula-for-derivative-of-path-of-Lagrangians}.
    
    \item Let $\zcal, \widetilde{\zcal}:(-\epsilon, \epsilon)$ be two paths through the same base point $\zcal_0 = \widetilde{\zcal}_0$ with smooth liftings $f_t,\widetilde{f_t}:L \rightarrow X$ through the same immersion $f_0 = \widetilde{f_0}$ representing $\zcal_0 = \widetilde{\zcal}_0$. The $L^2$ inner product of corresponding tangent 1-forms defines an inner product on $T_{\zcal_0} \scal \lcal$ 
    \[ \left\langle \dot{\zcal}(0), \dot{\widetilde{\zcal}}(0) \right\rangle_{\scal\lcal} :=  \langle \theta_{\zcal, f_0}, \theta_{\widetilde{\zcal}, f_0} \rangle_{L^2(f_0^* g)} = \int_L f_0^*g(\theta_{\zcal, f_0}, \theta_{\widetilde{\zcal}, f_0}) d\vol_{f_0^*g}. \]
    $\left\langle \cdot, \cdot \right\rangle_{\scal\lcal} $ is independent of the choice $f_0:L \rightarrow X$ by Proposition~\ref{prop:HWB-L2-inner-prod-well-defined}, making it well-defined.
\end{enumerate}
\end{definition}

\begin{proposition}
\label{prop:HWB-L2-inner-prod-well-defined}
    The $L^2$ inner product from Definition~\ref{def:HWB-L2-iner-product}.2 is independent of the choice of lifting $f_0:L \rightarrow X$.
\end{proposition}
\begin{proof}
\begin{lemma}
\label{lemma:HWB-akveld-salamon-change-of-lifting}
    Let $f_t,\overline{f}_t:L \rightarrow X$ be smooth paths lifting a given smooth path $\zcal: (-\epsilon, \epsilon) \rightarrow \scal\lcal$. There exists some $\psi_t \in \Diff(L)$ such that $\overline{f}_t = f_t$ for all $t\in (-\epsilon, \epsilon)$. Then
    \[\theta_{\zcal_0, \overline{f}_0} = \psi_0^* \theta_{\zcal_0, {f}_0}\]
\end{lemma}
\begin{proof}
    Observe that 
       \[\frac{d\overline{f}_t}{dt} = \frac{df_t}{dt}\circ \psi_t + df_t(\frac{d\psi_t}{dt}).\]
       Since $f_t$ is a Lagrangian immersion $f_t^*\omega = 0$, and so for any $\xi \in TL$,
      \[\begin{aligned}
          \overline{f}_t^* i_{\frac{d\overline{f}_t}{dt}} \omega (\xi) =&\  \overline{f}_t^* i_{\frac{df_t}{dt}\circ \psi_t} \omega (\xi) + \overline{f_t}^* i_{df_t(\frac{d\psi_t}{dt})}\omega (\xi) \\
          =&\ \overline{f}_t^* i_{\frac{df_t}{dt}\circ \psi_t} \omega (\xi) + \omega\left(df_t(\frac{d\psi_t}{dt}), d\overline{f}_t(\xi) \right) \\
          =&\ \overline{f}_t^* i_{\frac{df_t}{dt}\circ \psi_t} \omega (\xi) + \omega\left(df_t(\frac{d\psi_t}{dt}), d{f}_t \circ d\psi_t(\xi) \right) \\
          =&\  \overline{f}_t^* i_{\frac{df_t}{dt}\circ \psi_t} \omega (\xi) + (f_t^*\omega)\left(\frac{d\psi_t}{dt}, d\psi_t(\xi) \right) \\
          =&\ \overline{f}_t^* i_{\frac{df_t}{dt}\circ \psi_t} \omega (\xi),
      \end{aligned}
      \]
      \begin{equation}
          \implies \overline{f}_0^*\left(\left.i_{\frac{d\overline{f}_t}{dt}} \omega\right|_{t=0} \right) = \psi_0^* f_0^*\left(\left.i_{\frac{df_t}{dt}} \omega\right|_{t=0} \right). \tag*{qedhere}
      \end{equation}
\end{proof}
Let $\psi_0 \in \Diff(L)$ and let $ \bar{f}_0 = f_0 \circ \psi_0$ denote another representative of the same immersed Lagrangian $\zcal_0$. Then by Lemma~\ref{lemma:HWB-akveld-salamon-change-of-lifting},
    \begin{align*}
        \langle \theta_{\zcal,  \bar{f}_0}, \theta_{\widetilde{\zcal},  \bar{f}_0} \rangle_{L^2(L,  \bar{f}_0^* g)} =&\ \int_L  \bar{f}_0^*g(\theta_{\zcal,  \bar{f}_0}, \theta_{\widetilde{\zcal},  \bar{f}_0}) d\vol_{ \bar{f}_0^*g} \\
        =&\  \int_L \psi_0^*f_0^*g(\psi_0^* \theta_{\zcal, f_0}, \ \psi_0^*\theta_{\widetilde{\zcal}, f_0}) d\vol_{\psi_0^* f_0^*g} \\
        =&\ \int_L \psi_0^*\Big(f_0^*g( \theta_{\zcal, f_0}, \theta_{\widetilde{\zcal}, f_0}) \Big) \psi_0^* d\vol_{f_0^*g} \\
        =&\ \int_L f_0^*g(\theta_{\zcal, f_0}, \theta_{\widetilde{\zcal}, f_0}) d\vol_{f_0^*g} \\
        =&\ \langle \theta_{\zcal, f_0}, \theta_{\widetilde{\zcal}, f_0} \rangle_{L^2(L, f_0^* g)}.\tag*{\qedhere}
    \end{align*}
\end{proof}

\section{Lagrangian flux}
\label{sec:HWB-lagrangian-flux}

This section is concerned with the proofs of Theorems~\ref{thm:HWB-rel-flux-well-defined} and \ref{thm:HWB-dual-flux-well-defined}. The relative and special Lagrangian flux functional are used in Section \ref{sec:HWB-canonicalcoordinatesandaffinestructure} to construct generalise Hitchin's affine structures on the moduli space of special Lagrangians \cite[p.~507-508]{hitchin1997moduli}.

\subsection{Relative Lagrangian flux}
\label{subsection:relativelagrangianflux}

This section deals with the relative Lagrangian flux functional \eqref{eq:HWB-intro-rel-flux} introduced by Solomon--Yuval \cite[{\S}4.2]{SolomonYuval2020-geodesics}. The various claims of Theorem~\ref{thm:HWB-rel-flux-well-defined} are proved as Lemmas in this section, and the Lemmas are combined to prove Theorem~\ref{thm:HWB-rel-flux-well-defined} in Section~\ref{subsec:HWB-proof-of-thm-1-2}.

\begin{lemma}
\label{lemma:HWB-rel-flux-well-defined-lemma-1}
    For any path smooth path $\zcal: [0,1]\rightarrow \lcal(X,L;\Lambda_1,\dots, \Lambda_d)$ and smooth lifting $f:[0,1] \times L \rightarrow X$, the quantity
    \[RF(f):= \left[  \int_0^1 f_t^* i_{\frac{df_t}{dt}} \omega dt \right]\]
    is well defined as an element of $H^1(L, \partial L;\RR)$.
\end{lemma}
\begin{proof}
    $f_t^* i_{\frac{df_t}{dt}} \omega $ is a closed 1-form vanishing on the boundary by Lemma~\ref{lemma:HWB-formula-for-derivative-of-path-of-Lagrangians}.1, so $RF(f)$ is a well defined element of $H^1(L,\partial L; \RR)$.
\end{proof}

An alternate way to characterize the cohomology class $RF(f)$ is by specifying its integral along every $C^1$ curves with end points on the boundary.

\begin{lemma}
    \label{lemma:HWB-formula-for-rel-flux-of-curve}
    Let $\zcal:[0,1]\rightarrow  \lcal(X, L; \Lambda_1,\dots, \Lambda_d)$ be a smooth path and $f : [0,1]\times L \rightarrow X$ respectively be a smooth lifting (Definition \ref{def:smoothpathoflagrangians}). Then for any $C^1$ curve $\gamma:[0,1] \rightarrow L$ satisfying $\gamma(0), \gamma(1) \in \partial L$,
    \[ \int_{\gamma} RF(f) = \int_{[0,1]^2} l_{f, \gamma}^* \omega, \]
    where $l_{f, \gamma}: [0,1]^2 \rightarrow X$ is defined by
    \[ l_{f, \gamma}(t,s) = f(t,\gamma(s)). \]
\end{lemma}

\begin{remark}
\label{remark:HWB-formula-for-rel-flux-of-curve}
    Recall that $H^1(L,\partial L; \RR)$ is the vector space dual to the relative homology $H_1(L, \partial L; \RR)$, and that $H^1(L,\partial L; \RR)$ is generated by the classes of $C^1$ curves $\gamma$ on $L$ whose end points lie on the boundary $\partial L$. In particular two elements of $H^1(L, \partial L;\RR)$ are equal if and only if their integral along any such curve $\gamma$ are equal. Lemma~\ref{lemma:HWB-formula-for-rel-flux-of-curve} gives an alternate characterization of $RF(f)$.
\end{remark}

\begin{proof}[Proof of Lemma~\ref{lemma:HWB-formula-for-rel-flux-of-curve}]
    By \eqref{eq:HWB-intro-rel-flux}
    \begin{align*}
        \int_{\gamma} RF(f) =&\ \int_{\gamma} \int_0^1 f_t^* i_{\frac{df_t}{dt}} \omega dt \\
        =&\ \int_{0}^1 \int_0^1 f_t^* i_{\frac{df_t}{dt}} \omega (\gamma'(s))dtds \\
        =&\ \int_{0}^1 \int_0^1 \omega \left( \frac{df_t}{dt}(\gamma(s)),  df_t(\gamma'(s)) \right) dtds \\
        =&\ \int_{0}^1 \int_0^1 \omega \left( \frac{\partial l_{f,\gamma}}{\partial t},  \frac{\partial l_{f,\gamma}}{\partial s}) \right) dtds \\
        =&\ \int_{[0,1]^2} l_{f, \gamma}^* \omega. \tag*{\qedhere}
    \end{align*}
\end{proof}

Lemma~\ref{lemma:HWB-rel-flux-path-homotopy-lemma} is the main step in the proof of Theorem~\ref{thm:HWB-rel-flux-well-defined}, and asserts that the relative Lagrangian flux of a path of immersed Lagrangians with boundary depends only on the end point preserving homotopy class of the path.

\begin{lemma}
\label{lemma:HWB-rel-flux-path-homotopy-lemma}
    Consider two smooth paths $\zcal, \bar{\zcal}:[0,1] \rightarrow \lcal(X, L; \Lambda_1,\dots, \Lambda_d)$ with equal end points, $\zcal_0 = \bar{\zcal}_0$ and $\zcal_1 = \bar{\zcal}_1$, that are smoothly homotopic with fixed end-points. Let $f,\bar{f}: [0,1]\times L \rightarrow X$ respectively be smooth liftings of $\zcal$ and $\bar{\zcal}$ (Definition \ref{def:smoothpathoflagrangians}), and assume that $f_0 = \bar{f}_0$. Then
    \[ RF(\bar{f}) = \left[ \int_0^1 \bar{f}_t^* i_{\frac{d\bar{f}_t}{dt}} \omega dt \right] =  \left[ \int_0^1 f_t^* i_{\frac{df_t}{dt}} \omega dt \right] = RF(f).\]
\end{lemma}
\begin{proof}
    Consider a smooth, end-point fixing homotopy $\hcal:[0,1]^2 \rightarrow \lcal(X,L; \Lambda_1,\dots, \Lambda_d)$ such that
    \begin{equation}
        \label{eq:HWB-rel-flux-homotopy}
        \begin{aligned}
        \hcal(0,t) = \zcal_t , \qquad \qquad \quad &\hcal(1,t)  = \bar{\zcal}_t, \\
        \quad \hcal(u,0) = \zcal_0 = \bar{\zcal}_0, \qquad \&\quad &\hcal(u,1) = \zcal_1 = \bar{\zcal}_1, \qquad \forall\ t,u\in [0,1].
    \end{aligned}
    \end{equation}
    Let $F:[0,1]^2\times L \rightarrow X$, $F_{u,t}:= F(u,t,\cdot))$ be a smooth family of free Lagrangian immersions lifting $\hcal$ (Definition \ref{def:smoothpathoflagrangians}). Note that $F(0,t)$ and $F(1,t)$ are respectively liftings of the smooth paths $\zcal_t$ and $\bar{\zcal}_t$ by definition, but they might not coincide with the given paths $f_t$ and $\bar{f}_t$. However by by Claim~\ref{claim:HWB-rel-flux-claim1} and the assumption $f_0 = \bar{f_0}$ the lifting $F$ can be chosen to satisfy
    \begin{equation}
        \label{eq:HWB-rel-flux-homotopy-lift-properties}
        F_{u,0}  = f_0 \quad\forall\ u\in [0,1], \quad \&\quad  F_{0,t} = f_t,\ F_{1,t} = \bar{f}_t \quad  \forall\ t\in [0,1].
    \end{equation}
    Finally since  $\hcal(u,1) = \zcal_1 = \bar{\zcal}_1$ for all $u\in [0,1]$, the Lagrangian immersions $F_{u,1}$ all represent the same immersed Lagrangian with boundary $\zcal_1$. So by Definition~\ref{def:immersionsandimmersedforms}, there exists a smooth path $\psi_u \in \Diff(L)$ such that
    \begin{equation}
        \label{eq:HWB-rel-flux-homotopy-lift-properties-2}
        F_{u,1}  = f_1 \circ \psi_u \qquad \forall\ u\in [0,1].
    \end{equation}
    \begin{claim}
    \label{claim:HWB-rel-flux-claim1}
        Let $F:[0,1]^2 \times L \rightarrow X$ denote a lifting of a smooth map $\hcal$ satisfying \eqref{eq:HWB-rel-flux-homotopy}. The lifting $F$ can be chosen to ensure that 
        \begin{equation*}
        F_{u,0}  = f_0 \quad\forall\ u\in [0,1], \quad \&\quad  F_{0,t} = f_t,\ F_{1,t} = \bar{f}_t \quad  \forall\ t\in [0,1].
    \end{equation*}
    \end{claim}
    \begin{proof}
        Since $F_{u,0}$ is a Lagrangian immersion representing the same immersed Lagrangian $\hcal(u,0) = \zcal_0 = \bar{\zcal}_0$, there exists a smooth path $\eta_u \in \Diff(L)$ (Definition \ref{def:immersionsandimmersedforms}) such that
        \[ F_{u,0} = f_0 \circ \eta_u \qquad \forall\ u \in[0,1]. \]
        Replace $F_{u,t}$ by $F_{u,t} \circ \eta_u^{-1}$ to assume without loss of generality that
        \[F_{u,0} = f_0 \forall\ u\in [0,1]. \]
        Since $F_{0,t} $ and $f_t$ represent the same immersed Lagrangian, there exists a smooth path $\psi_t \in \Diff(L)$ such that
        \[F_{0,t} = f_t \circ \psi_t \qquad \forall\ t\in [0,1].\]
        Since $F_{0,0} = f_0$ and because the immersions are free, $\psi_0 = \id: L \rightarrow L$ (Definition~\ref{def:immersionsandimmersedforms}). Similarly since $F_{1,t} $ and $\bar{f}_t$ represent the same immersed Lagrangian, there exists a smooth path $\bar{\psi}_t \in \Diff(L)$ such that \[F_{1,t} = \bar{f}_t \circ \bar{\psi}_t \qquad \forall\ t\in [0,1].\]
        Since $F_{1,0} = f_0 = \bar{f}_0$, conclude that $\bar{\psi}_0 = \id$. Consider the modified lifting
        \[ \widetilde{F}_{u,t} := F_{u,t} \circ \psi_{(1-u)t}^{-1} \circ \bar{\psi}_{ut}^{-1}. \]
        Since $\psi_0 = \bar{\psi}_0 = \id$, it follows that $\widetilde{F}_{u,0}  = f_0$ for all $u\in [0,1]$, and that $\widetilde{F}_{0,t} = f_t $, $\widetilde{F}_{1,t} = \bar{f}_t$ for all $t\in [0,1]$.
    \end{proof}

    Fix a $C^1$ curve $\gamma:[0,1]\rightarrow L$ such that $\gamma(0), \gamma(1) \in \partial L$. Since $C_1,\dots, C_d$ are the components of $\partial L$, assume WLOG that $\gamma(0)\in C_i$ and $\gamma(1) \in C_j$ for some $1\le i,j\le d$. Consider
    \begin{equation}
        \label{eq:HWB-rel-flux-lu-def}
        l_{u}:[0,1]^2 \rightarrow X, \qquad l_u(t,s) := F(u,t,\gamma(s)) \qquad \forall\ u\in [0,1],
    \end{equation}
    where $F:[0,1]^2 \times L \rightarrow X$ is a lifting of $\hcal$ \eqref{eq:HWB-rel-flux-homotopy} satisfying \eqref{eq:HWB-rel-flux-homotopy-lift-properties} and \eqref{eq:HWB-rel-flux-homotopy-lift-properties-2}. By Lemma~\ref{lemma:HWB-formula-for-rel-flux-of-curve} and \eqref{eq:HWB-rel-flux-homotopy-lift-properties},
    \begin{equation}
        \label{eq:DWB-rel-flux-main-objective-pair}
        \int_{\gamma} RF(f) = \int_{[0,1]^2} l_0^*\omega \qquad \& \qquad \int_{\gamma} 
 RF(\bar{f}) = \int_{[0,1]^2} l_1^*\omega.
    \end{equation}
    By Remark \ref{remark:HWB-formula-for-rel-flux-of-curve}, if the two integrals of \eqref{eq:DWB-rel-flux-main-objective-pair} are equal for any choice of $\gamma$, then $RF(f) = RF(\bar{f})$. A heuristic for equality of the integrals is provided in Figure~\ref{fig:HWB-rel-flux-homotopy-independent}.

    \begin{figure}[t]
        \centering
        \includegraphics[width=0.8\linewidth]{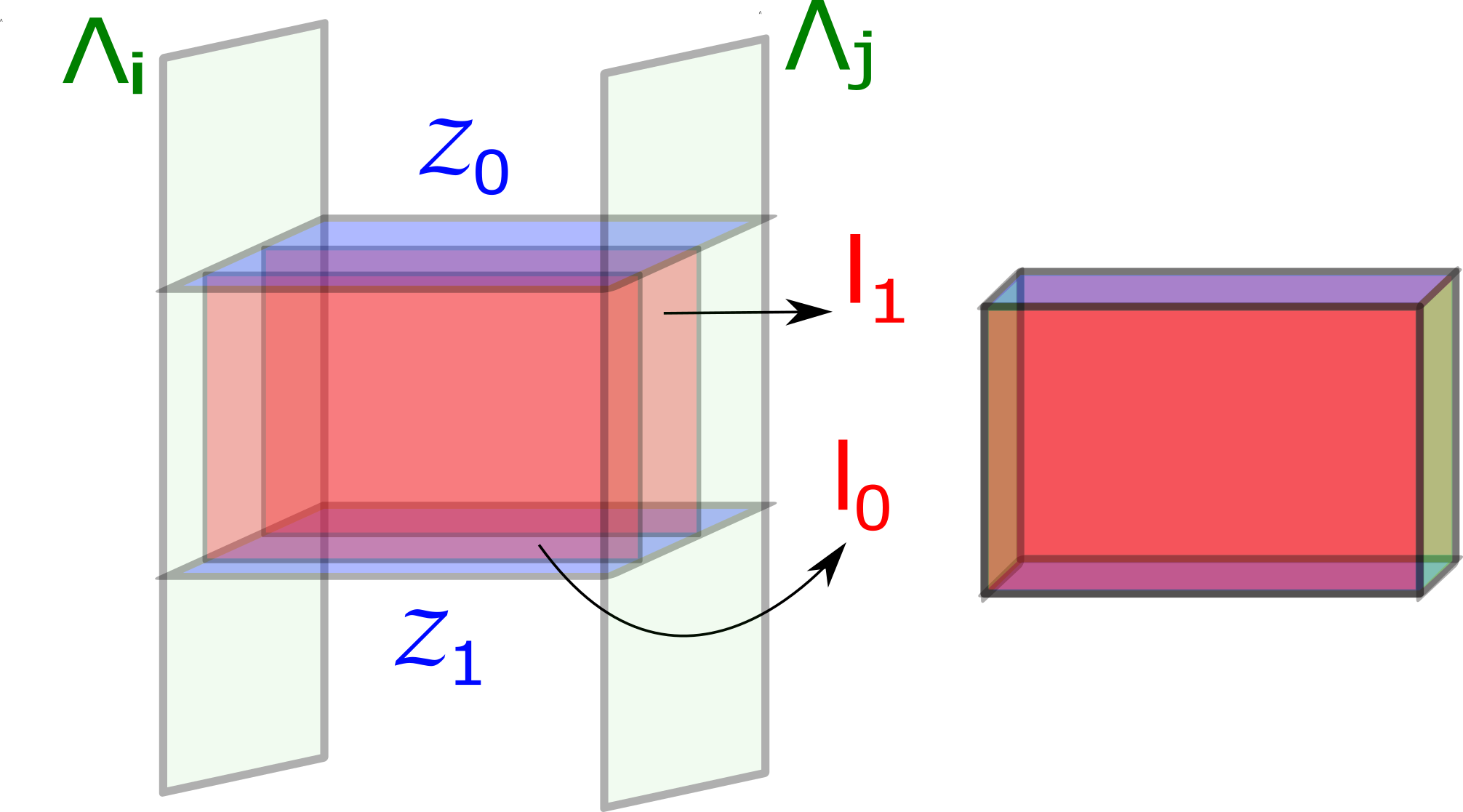}
        \caption{\textbf{Left:} The boundary Lagrangians $\Lambda_i,\Lambda_j$ (vertical planes), the immersed Lagrangian end points $\zcal_0, \zcal_1$ (horizontal planes), and the images of the squares $l_0,l_1$ depicted as rectangles parallel to the page. \textbf{Right:} The integral of $\omega$ on the surface of the box formed by $\Lambda_0, \Lambda_1, \zcal_0, \zcal_1, l_0([0,1]^2)$ and $l_1([0,1]^2)$ is zero since $\omega$ is closed. The top, bottom, left and right side of the cube are Lagrangian, and so the integral of $\omega$ over them vanishes. Thus the integral of $\omega$ over the front and back sides has opposite sign when oriented with the outward normal.}
        \label{fig:HWB-rel-flux-homotopy-independent}
    \end{figure}

    \begin{claim}
    \label{claim:HWB-rel-flux-well-defined-main-claim}
    The real valued function
        \[u\mapsto \int_{[0,1]^2} l_u^* \omega\]
        is independent of $u\in[0,1]$.
    \end{claim}
    \begin{proof}
        Differentiating in $u$ and applying Cartan's formula for the Lie derivative,
        \[ \begin{aligned}
            \frac{d}{du} \int_{[0,1]^2} l_u^* \omega =&\ \int_{[0,1]^2} l_u^* \left( di_{\frac{dl_u}{du}} \omega + i_{\frac{dl_u}{du}}d \omega \right) \\
            =&\ \int_{[0,1]^2} l_u^* \left( di_{\frac{dl_u}{du}} \omega  \right) \\
            =&\ \int_{\partial ([0,1]^2)} l_u^* \left(i_{\frac{dl_u}{du}} \omega  \right),
        \end{aligned} \]
        where the second line follows since $\omega$ is closed and the third line from an application of Stokes' theorem. Note that \[\partial [0,1]^2 = \{0\}\times [0,1] \cup \{1\}\times [0,1] \cup [0,1]\times\{0\} \cup [0,1]\times\{1\}. \]
        So conclude that
        \[\begin{aligned}
            \frac{d}{du} \int_{[0,1]^2} l_u^* \omega =&\ \int_{\partial ([0,1]^2)} l_u^* \left(i_{\frac{dl_u}{du}} \omega  \right) \\
            =&\ \int_{\{0\}\times [0,1]} l_u^* \left(i_{\frac{dl_u}{du}} \omega  \right)  + \int_{ \{1\}\times [0,1] } l_u^* \left(i_{\frac{dl_u}{du}} \omega  \right)\\
            &\ + \int_{[0,1]\times\{0\}} l_u^* \left(i_{\frac{dl_u}{du}} \omega  \right) + \int_{[0,1]\times\{1\}} l_u^* \left(i_{\frac{dl_u}{du}} \omega  \right)
        \end{aligned}\]
        \begin{equation}
            \label{eq:HWB-rel-flux-claim-final-conclusion-1}
            \begin{aligned}
                \implies \quad \frac{d}{du} \int_{[0,1]^2} l_u^* \omega =&\ \int_0^1 \omega \left( \frac{dl_u}{du}(0,s), \frac{\partial l_u}{\partial s}(0,s) \right) ds\\
                &\ +  \int_0^1  \omega \left( \frac{dl_u}{du}(1,s), \frac{\partial l_u}{\partial s}(1,s) \right) ds \\
            &\ +  \int_0^1  \omega \left( \frac{dl_u}{du}(t,0), \frac{\partial l_u}{\partial t}(t,0) \right) dt \\
            &\ +  \int_0^1  \omega \left( \frac{dl_u}{du}(t,1), \frac{\partial l_u}{\partial t}(t,1) \right) dt.
            \end{aligned}
        \end{equation}
        Each of the integrands in the four terms on the RHS of \eqref{eq:HWB-rel-flux-claim-final-conclusion-1} vanishes:
        \begin{itemize}
            \item[(i) \& (ii):] By \eqref{eq:HWB-rel-flux-homotopy-lift-properties}, \eqref{eq:HWB-rel-flux-homotopy-lift-properties-2} and \eqref{eq:HWB-rel-flux-lu-def},
        \[l_u(0,s) = F(u,0,\gamma(s)) = f_0 (\gamma(s) ) \quad \& \quad l_u(1,s) = F(u,1,\gamma(s)) = f_1\circ \psi_u (\gamma(s)) \]
        \[\implies \quad \frac{dl_u}{du}(0,s) = 0 \quad \& \quad \frac{dl_u}{du}(1,s) = df_1\left( \frac{d\psi_u}{du}(\gamma(s))\right)\]
        \begin{equation}
            \label{eq:HWB-rel-flux-well-def-main-eq-1-1}
            \implies \quad \frac{dl_u}{du}(0,s)\in df_0(T^*_{\gamma(s)}L) \quad \&\quad  \frac{dl_u}{du}(1,s) \in df_1(T^*_{\psi_u\circ \gamma(s)}L).
        \end{equation}
        Similarly
        \[\frac{\partial l_u}{\partial s}(0,s) = df_0(\gamma'(s)) \quad \& \quad \frac{\partial l_u}{\partial s}(1,s) = df_1 \circ d\psi_u(\gamma'(s)) \]
        \begin{equation}
            \label{eq:HWB-rel-flux-well-def-main-eq-1-2}
            \implies \quad \frac{\partial l_u}{\partial s} (0,s)\in df_0(T^*_{\gamma(s)}L) \quad \&\quad \frac{\partial l_u}{\partial s} (1,s) \in df_1(T^*_{\psi_u\circ \gamma(s)}L).
        \end{equation}
        But since $f_0$ and $f_1$ are Lagrangian immersions, $df_0(T^*_{\gamma(s)}L)$ and $ df_1(T^*_{\psi_u\circ \gamma(s)}L)$ are Lagrangian planes. Combining this with \eqref{eq:HWB-rel-flux-well-def-main-eq-1-1} and \eqref{eq:HWB-rel-flux-well-def-main-eq-1-2},
        \begin{equation}
            \label{eq:HWB-rel-flux-claim-final-conclusion-2}
            \omega \left( \frac{dl_u}{du}(0,s), \frac{\partial l_u}{\partial s}(0,s) \right) = 0 \quad \&\quad  \omega \left( \frac{dl_u}{du}(1,s), \frac{\partial l_u}{\partial s}(1,s) \right) = 0 \qquad \forall\ s\in [0,1].
        \end{equation}

        \item[(iii) \& (iv):] By \eqref{eq:HWB-rel-flux-homotopy-lift-properties}, \eqref{eq:HWB-rel-flux-homotopy-lift-properties-2} and \eqref{eq:HWB-rel-flux-lu-def},
        \[l_u(t,0) = F_{u,t}(\gamma(0))  \quad \& \quad l_u(1,s) = F_{u,t}(\gamma(1)). \]
        
        Since $F_{u,t}$ is an element of $\lcal(X,L;\Lambda_1,\dots, \Lambda_d)$, by Definition \ref{def:HWB-boundary-conditions-on-immersions-and-lagrangians} and the assumption $\gamma(0)\in C_i$ and $\gamma(1) \in C_j$, $l_u(t,0) = F_{u,t}(\gamma(0)) \in \Lambda_i$ and $l_u(t,1) = F_{u,t}(\gamma(1)) \in \Lambda_j$ for all $t\in [0,1]$,
        \begin{equation}
            \label{eq:HWB-rel-flux-well-def-main-eq-2-1}
            \implies \quad \frac{d l_u}{du} (t,0),\  \frac{\partial l_u}{\partial t} (t,0)\ \in\  T\Lambda_i \quad \&\quad \frac{d l_u}{du} (t,1),\  \frac{\partial l_u}{\partial t} (t,1)\ \in\  T\Lambda_j.
        \end{equation}
        But $\Lambda_i, \Lambda_j$ are Lagrangian submanifolds. Conclude by \eqref{eq:HWB-rel-flux-well-def-main-eq-2-1} that
        \begin{equation}
            \label{eq:HWB-rel-flux-claim-final-conclusion-3}
            \omega \left( \frac{dl_u}{du}(t,0), \frac{\partial l_u}{\partial t}(t,0) \right) =0 \quad \& \quad   \omega \left( \frac{dl_u}{du}(t,1), \frac{\partial l_u}{\partial t}(t,1) \right) =0 \qquad \forall\ t\in [0,1].
        \end{equation}
        \end{itemize}
        
        Claim~\ref{claim:HWB-rel-flux-well-defined-main-claim} follows from \eqref{eq:HWB-rel-flux-claim-final-conclusion-1}, \eqref{eq:HWB-rel-flux-claim-final-conclusion-2} and \eqref{eq:HWB-rel-flux-claim-final-conclusion-3}.
    \end{proof}

    Lemma~\ref{lemma:HWB-rel-flux-path-homotopy-lemma} follows from Remark \ref{remark:HWB-formula-for-rel-flux-of-curve}, \eqref{eq:DWB-rel-flux-main-objective-pair} and Claim~\ref{claim:HWB-rel-flux-well-defined-main-claim}.
\end{proof}

\subsection{Proof of Theorem~\ref{thm:HWB-rel-flux-well-defined}}
\label{subsec:HWB-proof-of-thm-1-2}

\begin{proof}[Proof of Theorem~\ref{thm:HWB-rel-flux-well-defined}]
    By Lemma~\ref{lemma:HWB-rel-flux-well-defined-lemma-1} $RF(f)$ is well defined, and by Lemma~\ref{lemma:HWB-rel-flux-path-homotopy-lemma} it only depends on the chosen lifting $f_0:L \rightarrow X$ of the end point $\zcal_0$, and the end point preserving homotopy class of $\zcal$ in $\lcal(X,L; \Lambda_1,\dots , \Lambda_d)$.
\end{proof}

\subsection{Special Lagrangian flux}
\label{subsection:duallagrangianflux}

This section deals with the special Lagrangian flux functional \eqref{eq:HWB-intro-rel-flux} and a proof of Theorem~\ref{thm:HWB-rel-flux-well-defined}. The proof of Theorem~\ref{thm:HWB-dual-flux-well-defined} is nearly identical to that of Theorem~\ref{thm:HWB-rel-flux-well-defined}. The various claims of Theorem~\ref{thm:HWB-rel-flux-well-defined} are proved as Lemmas in this section, analogous to the results in Section~\ref{subsection:relativelagrangianflux}. The Lemmas are combined to prove Theorem~\ref{thm:HWB-dual-flux-well-defined} in Section~\ref{subsec:HWB-proof-of-thm-2-2}.

\begin{remark}
    The relative Lagrangian flux is a functional on paths of Lagrangians, i.e., paths of $n$-dim submanifolds on which the symplectic form $\omega$ vanishes. Similarly the special Lagrangian flux is naturally a functional on paths of $n$-dim submanifolds on which $\Im(\Omega)$ vanishes. For simplicity the discussion is restricted to paths of special Lagrangian submanifolds on which both $\omega$ and $\Im(\Omega)$ vanish, but the same proof carries over more generally when only $\Im(\Omega)$ vanishes.
\end{remark}

\begin{lemma}
\label{lemma:HWB-dual-flux-well-defined-lemma-1}
    For any path smooth path $\zcal: [0,1]\rightarrow \scal \lcal(X,L;\Lambda_1,\dots, \Lambda_d)$ and any smooth lifting $f:[0,1] \times L \rightarrow X$ of the path (Definition \ref{def:smoothpathoflagrangians}),
    \[SF(f):= \left[  \int_0^1 f_t^* i_{\frac{df_t}{dt}} \Im(\Omega) dt \right] \]
    is well defined as an element of $H^{n-1}(L;\RR)$.
\end{lemma}
\begin{proof}
    Follows by Lemma~\ref{lemma:HWB-phit-is-closed}.
\end{proof}

Lemma~\ref{lemma:HWB-formula-for-dual-flux-of-curve} is analogous to Lemma~\ref{lemma:HWB-formula-for-rel-flux-of-curve}, and is an alternate characterization of $SF(f)$.

\begin{lemma}
    \label{lemma:HWB-formula-for-dual-flux-of-curve}
    Let $\zcal:[0,1]\rightarrow \scal \lcal(X, L; \Lambda_1,\dots, \Lambda_d)$ be a smooth path and $f : [0,1]\times L \rightarrow X$ be a smooth lifting (Definition \ref{def:smoothpathoflagrangians}). Then for any closed $(n-1)$-manifold $B$ and any $C^2$ map $\sigma:B \rightarrow L$,
    \[ SF(f) \cdot \sigma^*[B] = \int_{B} \sigma^* SF(f) = \int_{[0,1]\times B} \beta^* \Im(\Omega), \]
    where $\beta: [0,1]\times B \rightarrow X$ depends on $f_t$ and $\sigma$ and is defined by
    \[ \beta_{f, \sigma} (t,q) := f(t,\sigma (q)). \]
\end{lemma}

\begin{remark}
    \label{remark:HWB-formula-for-dual-flux-of-curve}
    Since $L$ is compact, by Thom's solution to the Steenrod problem, $H_{n-1}(L; \RR)$ is generated by the images of fundamental classes of compact manifolds, i.e., by elements of the form $\sigma_*([B])$ where $\sigma:B \rightarrow L$ a smooth map into $L$ from a closed oriented $(n-1)$-dimensional manifold $B$ \cite[Theorem III.4]{Thom1954}. The reader is referred to a survey by Sullivan \cite{Sullivan2004} additional details. In particular, Lemma~\ref{lemma:HWB-formula-for-dual-flux-of-curve} is an alternate characterization of the class $SF(f)$.
\end{remark}

\begin{proof}[Proof of Lemma~\ref{lemma:HWB-formula-for-dual-flux-of-curve}]
    Using the definition from \eqref{eq:HWB-intro-dual-flux},
    \begin{equation}
    \label{eq:HWB-formula-for-dual-flux-of-curve-eq1}
        \begin{aligned}
        \int_{B} \sigma^* SF(f) =&\ \int_{B} \sigma^* \int_0^1 f_t^* i_{\frac{df_t}{dt} \circ \sigma} \Im(\Omega) dt \\
        =&\ \int_0^1 \left( \int_B  \sigma^* f_t^* i_{\frac{df_t}{dt} \circ \sigma} \Im(\Omega)  \right)dt.
    \end{aligned}
    \end{equation}
    Define $\beta_t := \beta(t,\cdot) =  f_t \circ \sigma:B \rightarrow X$. Since $\beta^* \Im(\Omega)$ is a top form on the product space $[0,1]\times B$, it is decomposable and
    \[ \begin{aligned}
        \beta^* \Im(\Omega) =&\ dt \wedge \left( i_{\frac{\partial}{\partial t}} \beta_t^* \Im(\Omega) \right) \\
        =&\ dt \wedge \ \sigma^* f_t^* i_{\frac{df_t}{dt} \circ \sigma} \Im(\Omega) 
        \end{aligned} \]
      Thus by \eqref{eq:HWB-formula-for-dual-flux-of-curve-eq1},  
    \begin{align*}
        \implies \quad  \int_{[0,1]\times B} \beta^* \Im(\Omega) =&\  \int_{[0,1]\times B} dt \wedge \ \sigma^* f_t^* i_{\frac{df_t}{dt} \circ \sigma} \Im(\Omega) \\
        =&\  \int_0^1 \left( \int_B  \sigma^* f_t^* i_{\frac{df_t}{dt} \circ \sigma} \Im(\Omega)  \right)dt\\
        =&\ \int_B \sigma^* SF(f).\tag*{\qedhere}
    \end{align*} 
\end{proof}

Lemma~\ref{lemma:HWB-dual-flux-path-homotopy-lemma} is analogous to Lemma~\ref{lemma:HWB-rel-flux-path-homotopy-lemma}, and is the main step in the proof of Theorem~\ref{thm:HWB-dual-flux-well-defined}. It asserts that the special Lagrangian flux of a path of immersed SLb depends only on the end point preserving homotopy class of the path.

\begin{lemma}
\label{lemma:HWB-dual-flux-path-homotopy-lemma}
    Consider two smooth paths $\zcal, \bar{\zcal}:[0,1] \rightarrow \scal\lcal(X, L; \Lambda_1,\dots, \Lambda_d)$ with equal end points, $\zcal_0 = \bar{\zcal}_0$ and $\zcal_1 = \bar{\zcal}_1$, that are smoothly homotopic with fixed end-points. Let $f,\bar{f}: [0,1]\times L \rightarrow X$ respectively be smooth liftings of $\zcal$ and $\bar{\zcal}$ (Definition \ref{def:smoothpathoflagrangians}), and assume that $f_0 = \bar{f}_0$. Then
    \[ SF(\bar{f}) = \left[ \int_0^1 \bar{f}_t^* i_{\frac{d\bar{f}_t}{dt}} \Im(\Omega) dt \right] =  \left[ \int_0^1 f_t^* i_{\frac{df_t}{dt}} \Im(\Omega) dt \right] = SF(f).\]
\end{lemma}

\begin{proof}
    Consider a smooth end-point fixed homotopy $\hcal:[0,1]^2 \rightarrow \scal \lcal(X,L; \Lambda_1,\dots, \Lambda_d)$ (Definition \ref{def:smoothpathoflagrangians}) such that
    \begin{equation}
        \label{eq:HWB-dual-flux-homotopy}
        \begin{aligned}
        \hcal(0,t) = \zcal_t , \qquad \qquad &\hcal(1,t)  = \bar{\zcal}_t, \\
        \quad \hcal(u,0) = \zcal_0 = \bar{\zcal}_0, \qquad \&\ &\hcal(u,1) = \zcal_1 = \bar{\zcal}_1, \qquad \forall\ t,u\in [0,1].
    \end{aligned}
    \end{equation}
    Let $F:[0,1]^2\times L \rightarrow X$, $F_{u,t}:= F(u,t,\cdot))$ be a smooth family of free special Lagrangian immersions lifting $\hcal$. Note that by the assumption $f_0 = \bar{f_0}$. Furthermore $F(0,t)$ and $F(1,t)$ are respectively liftings of the smooth paths $\zcal_t$ and $\bar{\zcal}_t$ by definition. In general $F(0,t)$ and $F(1,t)$ might not coincide with the given paths $f_t$ and $\bar{f}_t$, but by Claim~\ref{claim:HWB-rel-flux-claim1}, the lifting $F$ can be chosen to satisfy
    \begin{equation}
        \label{eq:HWB-dual-flux-homotopy-lift-properties}
        F_{u,0}  = f_0 \quad\forall\ u\in [0,1], \quad \&\quad  F_{0,t} = f_t,\ F_{1,t} = \bar{f}_t \quad  \forall\ t\in [0,1].
    \end{equation}
    Finally since  $\hcal(u,1) = \zcal_1 = \bar{\zcal}_1$ for all $u\in [0,1]$, the Lagrangian immersions $F_{u,1}$ all represent the same immersed Lagrangian with boundary $\zcal_1$. So by Definition \ref{def:immersionsandimmersedforms}, there exists a smooth path $\psi_u \in \Diff(L)$ such that
    \begin{equation}
        \label{eq:HWB-dual-flux-homotopy-lift-properties-2}
        F_{u,1}  = f_1 \circ \psi_u \qquad \forall\ u\in [0,1].
    \end{equation}

    Fix a closed $(n-1)$-manifold $B$ and a smooth map $\sigma : B \rightarrow L$, and define
    \begin{equation}
        \label{eq:HWB-dual-flux-lu-def}
        \beta_{u}:[0,1]\times B \rightarrow X, \qquad \beta_u(t,q) := F(u,t,\sigma(q)) \qquad \forall\ u\in [0,1].
    \end{equation}
    By Lemma~\ref{lemma:HWB-formula-for-dual-flux-of-curve} and assumption \eqref{eq:HWB-dual-flux-homotopy-lift-properties},
    \begin{equation}
        \label{eq:DWB-dual-flux-main-objective-pair}
        \int_{\gamma} SF(f) = \int_{[0,1]\times B} \beta_0^* \Im(\Omega) \quad \& \quad \int_{\gamma} SF(\bar{f}) = \int_{[0,1]\times B} \beta_1^* \Im(\Omega).
    \end{equation}
    By Remark \ref{remark:HWB-formula-for-dual-flux-of-curve}, if the two integrals of \eqref{eq:DWB-dual-flux-main-objective-pair} are equal for any choice of $B$ and $\sigma:B \rightarrow L$, then $SF(f) = SF(\bar{f})$.

    \begin{figure}[tbh]
        \centering
        \includegraphics[width=0.4\linewidth]{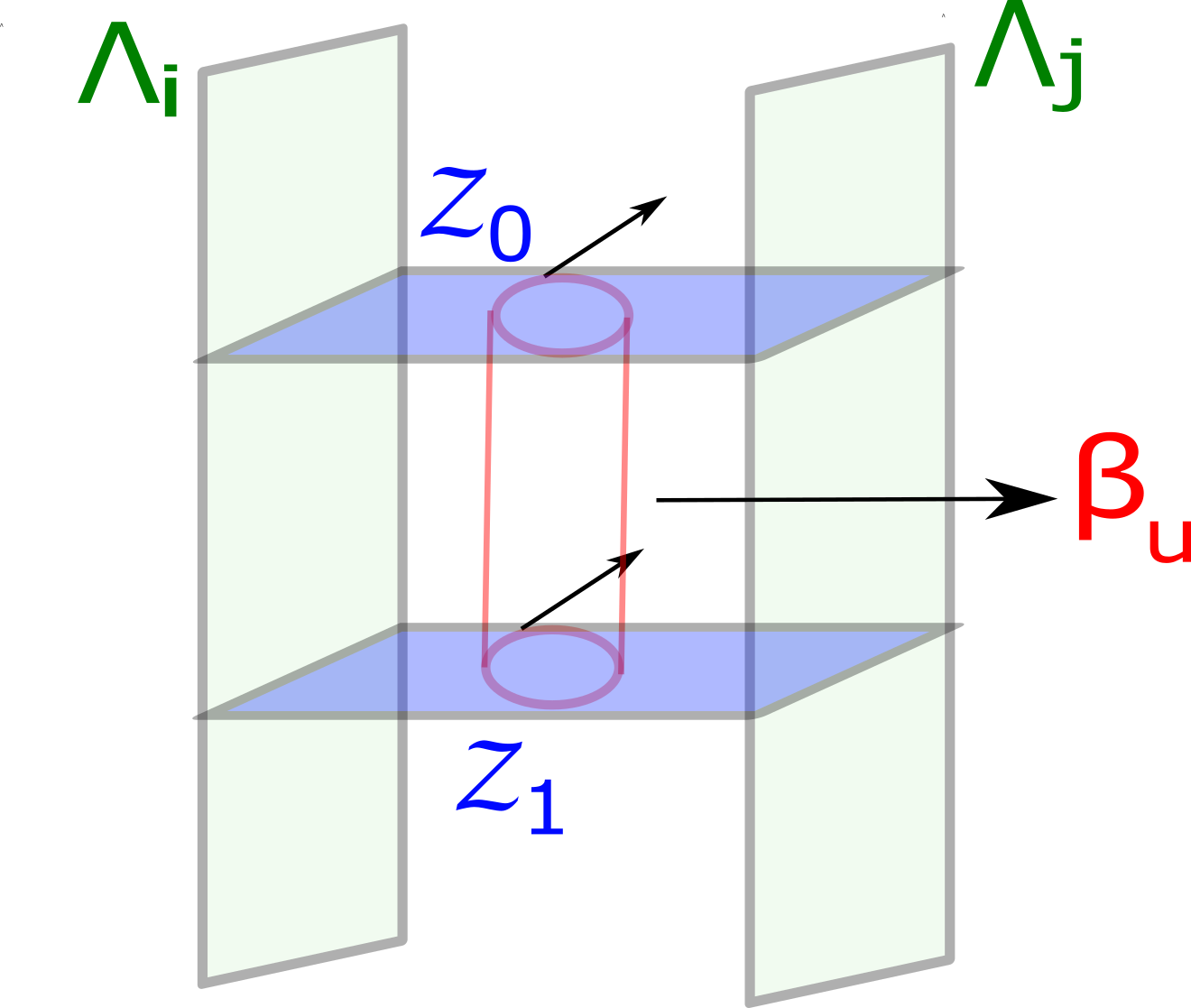}
        \caption{The boundary Lagrangians $\Lambda_i, \Lambda_j$ (vertical planes), the immersed special Lagrangian end points $\zcal_0, \zcal_1$ (horizontal planes), and the image of $[0,1]\times B$ under $\beta_u$ represented as a cylinder with bounadry on $\zcal_0$ and $\zcal_1$. The arrows depict $\frac{d\beta_{u}}{du}$ on $\{0\}\times B$ and $\{1\} \times B$, which are respectively tangent to $\zcal_0$ and $\zcal_1$.}
        \label{fig:HWB-dual-flux-homotopy-independent}
    \end{figure}

    \begin{claim}
    \label{claim:HWB-dual-flux-well-defined-main-claim}
    The real valued function
        \[u\mapsto \int_{[0,1]\times B} \beta_u^* \Im(\Omega)\]
        is independent of $u\in [0,1]$.
    \end{claim}
    \begin{proof}
        Differentiating in $u$ and applying Cartan's formula for the Lie derivative,
        \[ \begin{aligned}
            \frac{d}{du} \int_{[0,1] \times B} \beta_u^* \Im(\Omega) =&\ \int_{[0,1] \times B} \beta_u^* \left( di_{\frac{d\beta_u}{du}} \Im(\Omega) + i_{\frac{d\beta_u}{du}} d \Im(\Omega) \right) \\
            =&\ \int_{[0,1] \times B} \beta_u^* \left( di_{\frac{d\beta_u}{du}} \Im(\Omega)  \right) \\
            =&\ \int_{\partial ([0,1] \times B)} \beta_u^* \left(i_{\frac{d\beta_u}{du}} \Im(\Omega)  \right).
        \end{aligned} \]
        The second line follows since $\Im(\Omega)$, being the imaginary part of a holomorphic $(n,0)$-form, is closed. The third line follows from an application of Stokes' theorem. Since $B$ is a closed manifold,
        \[\partial \left( [0,1]\times B \right) = \{0\}\times B \bigcup \{1\}\times B, \]
        \begin{equation}
            \label{eq:HWB-dual-flux-claim-final-conclusion-1}
            \begin{aligned}
            \implies \quad \frac{d}{du} \int\displaylimits_{[0,1]\times B} \beta_u^* \Im(\Omega) =&\ \int\displaylimits_{\partial ([0,1] \times B)} \beta_u^* \left(i_{\frac{d\beta_u}{du}} \Im(\Omega)  \right) \\
            =&\ \int\displaylimits_{\{0\}\times B} \beta_u^* \left(i_{\frac{d\beta_u}{du}}  \Im(\Omega)  \right)  + \int\displaylimits_{ \{1\}\times B } \beta_u^* \left(i_{\frac{d\beta_u}{du}}  \Im(\Omega)  \right).
        \end{aligned}
        \end{equation}
        The integrands on the RHS of \eqref{eq:HWB-dual-flux-claim-final-conclusion-1} will turn out to be zero for every $u\in[0,1]$, because $\frac{d\beta_u}{du}$ is tangent to the $\zcal_0$ and $\zcal_1$ on the respective boundary components $\{0\}\times B$ and $\{1\}\times B$ by \eqref{eq:HWB-dual-flux-well-def-main-eq-1-1} (Figure~\ref{fig:HWB-dual-flux-homotopy-independent}). Combining \eqref{eq:HWB-dual-flux-homotopy-lift-properties}, \eqref{eq:HWB-dual-flux-homotopy-lift-properties-2} and \eqref{eq:HWB-dual-flux-lu-def},
            \begin{equation}
                \label{eq:HWB-dual-flux-well-def-main-eq-0--1}
                \begin{aligned}
                    \beta_u(0,q) =&\ F(u,0,\sigma(q)) = f_0 \circ \sigma(q) \\
                    \beta_u(1,q) =&\ F(u,1,\sigma(q)) = f_1\circ \psi_u \circ \sigma(q)
                \end{aligned}
                \qquad \forall\ q\in B,\ u\in [0,1],
            \end{equation}
        \begin{equation}
            \label{eq:HWB-dual-flux-well-def-main-eq-1-1}
            \implies \qquad \frac{d\beta_u}{du}(0,q) = 0 \qquad \& \qquad \frac{d\beta_u}{du}(1,q) = df_1\left( \frac{d\psi_u}{du}(\sigma(q))\right).
        \end{equation}
        By \eqref{eq:HWB-dual-flux-claim-final-conclusion-1}, \eqref{eq:HWB-dual-flux-well-def-main-eq-0--1}, and \eqref{eq:HWB-dual-flux-well-def-main-eq-1-1},
        \begin{align*}
            \frac{d}{du} \int_{[0,1]\times B} \beta_u^* \Im(\Omega) =&\ \int_{\{0\}\times B} \beta_u^* \left(i_{\frac{d\beta_u}{du}}  \Im(\Omega)  \right)  + \int_{ \{1\}\times B } \beta_u^* \left(i_{\frac{d\beta_u}{du}}  \Im(\Omega)  \right) \\
            =&\ \int_{ B} \sigma^* f_0^* \left(i_{ 0}  \Im(\Omega)  \right)  + \int_{ B }  \sigma^* \psi_u^* f_1^* \left(i_{ df_1\left( \frac{d\psi_u}{du} \right)}  \Im(\Omega)  \right) \\
            =&\ 0+ \int_{ B }  \sigma^* \psi_u^* \left(i_{\frac{d\psi_u}{du}}  f_1^* \Im(\Omega)  \right)  0,
        \end{align*}
        where the last line follows from the assumption $f_1^*\Im(\Omega) = 0$, the special Lagrangian condition on $f_1$. Claim~\ref{claim:HWB-dual-flux-well-defined-main-claim} follows.
    \end{proof}

    Lemma~\ref{lemma:HWB-dual-flux-path-homotopy-lemma} follows from Remark \ref{remark:HWB-formula-for-dual-flux-of-curve}, \eqref{eq:DWB-dual-flux-main-objective-pair} and Claim~\ref{claim:HWB-dual-flux-well-defined-main-claim}.
\end{proof}

\subsection{Proof of Theorem~\ref{thm:HWB-dual-flux-well-defined}}
\label{subsec:HWB-proof-of-thm-2-2}

\begin{proof}[Proof of Theorem~\ref{thm:HWB-dual-flux-well-defined}]
    By Lemma~\ref{lemma:HWB-dual-flux-well-defined-lemma-1} $SF(f)$ is well defined, and by Lemma~\ref{lemma:HWB-dual-flux-path-homotopy-lemma} it only depends on the chosen lifting $f_0:L \rightarrow X$ of the end point $\zcal_0$, and the end point preserving homotopy class of $\zcal$ in $ \scal \lcal(X,L; \Lambda_1,\dots , \Lambda_d)$.
\end{proof}

\section{Canonical coordinates and affine structures}
\label{sec:HWB-canonicalcoordinatesandaffinestructure}

This section contains a proof of Theorem~\ref{thm:HWB-affine-structure-on-moduli-space}, generalizing Hitchin's canonical affine structures \cite[p.~507-508]{hitchin1997moduli} to the moduli space of SLb. The relative Lagrangian flux \ref{eq:HWB-intro-rel-flux} and special Lagrangian flux \ref{eq:HWB-intro-dual-flux} each induces a family of affine charts on $\scal\lcal(X,L; \Lambda_1,\dots, \Lambda_d)$ taking values in $H^1(L, \partial L; \RR)$ and $H^{n-1}(L;\RR)$ respectively.

The various steps in the construction of affine structures are carried out in \S\ref{subsec:HWB-canonical-coordinates}, and Theorem~\ref{thm:HWB-affine-structure-on-moduli-space} is proved in \S\ref{subsec:HWB-proof-of-affine-structures-thm}. Finally the isometric and Lagrangian embedding of Proposition~\ref{prop:HWB-isometric-embedding-intro} is proved in \S\ref{subsec:HWB-isometric-and-lagrangian-embedding}.

For the rest of the section, assume that $X$ is a Calabi--Yau manifold of complex dimension $n$, $\Lambda_1,\dots, \Lambda_d \subset X$ are embedded Lagrangian subamnifolds, $L$ is an $n$-dim compact manifold with smooth boundary and $d$ boundary components, and $m := \dim H^1(L, \partial L; \RR) = \dim H^{n-1}(L;\RR)$. By Theorem~\ref{thm:HWB-tangentspacetospeciallagrangianswithboundary-2} $\scal\lcal(X,L; \Lambda_1,\dots, \Lambda_d)$ is an $m$-dim manifold.

\subsection{Canonical coordinates}
\label{subsec:HWB-canonical-coordinates}

This section contains a construction of the two families of coordinate charts mentioned in Theorem~\ref{thm:HWB-affine-structure-on-moduli-space}, associated respectively to the relative Lagrangian flux \eqref{eq:HWB-intro-rel-flux}  and special Lagrangian flux \eqref{eq:HWB-intro-dual-flux}. First step is to construct the families of local maps 

\begin{definition}
\label{def:HWB-def-of-coordinate-chart}
    Fix an immersed special Lagrangian $\zcal_0 \in \scal \lcal(X,L; \Lambda_1,\dots , \Lambda_d)$ and an immersion representing it $f_0: L \rightarrow X$. Let $\ucal \subset \scal \lcal$ be a simply-connected neighborhood of $\zcal_0$. Define two functionals
    \[R_{f_0}:\ucal \rightarrow H^1(L, \partial L; \RR) \quad \& \quad S_{f_0}:\ucal \rightarrow H^{n-1}(L, \RR), \]
    as follows. For any $\zcal_1 \in \ucal$, consider a smooth path $\zcal:[0,1]\rightarrow \ucal$ connecting $\zcal_0$ to $\zcal_1$, and a smooth lifting $f:[0,1] \times L \rightarrow X$ (Definition \ref{def:smoothpathoflagrangians}) such that $f(0,\cdot) = f_0$. Note that such a lifting always can always be chosen by pre composing any lifting with some element of $\Diff(L)$ (Definition \ref{def:immersionsandimmersedforms}). Define
    \begin{equation}
        \label{eq:HWB-def-of-coordinate-chart}
        R_{f_0}(\zcal_1) := RF(f) \qquad \&\qquad S_{f_0}(\zcal_1) := DF(f).
    \end{equation}
    By Theorems~\ref{thm:HWB-rel-flux-well-defined} and \ref{thm:HWB-dual-flux-well-defined} $R_{f_0}$ and $D_{f_0}$ are well defined since $\ucal$ is simply-connected and depend only on the chosen pointed smooth coordinate chart $(\ucal, \zcal_0, f_0, \Phi)$.
\end{definition}

\begin{lemma}
\label{lemma:HWB-rel-and-dual-are-local-diffeo}
    Fix an immersed special Lagrangian $\zcal_0 \in \scal \lcal(X,L; \Lambda_1,\dots , \Lambda_d)$ and an immersion representing it $f_0: L \rightarrow X$. Let $\ucal \subset \scal \lcal$ be a simply-connected neighborhood of $\zcal_0$. After shrinking $\ucal$ if necessary the associated maps $R_{f_0}:\ucal \rightarrow H^1(L, \partial L; \RR)$ and $S_{f_0}:\ucal \rightarrow H^{n-1}(L, \RR)$ from Definition \ref{def:HWB-def-of-coordinate-chart} are diffeomorphisms. 
\end{lemma}
\begin{proof}
    The strategy is to prove that the derivatives of $R_{f_0}$ and $S_{f_0}$ at the base point $\zcal_0$ are invertible, followed by an appeal to the inverse function theorem. Fix a basis $\theta_1,\dots \theta_m$ of 1-forms vanishing on the boundary, that are harmonic fields with respect to the pullback Riemannian metric $f^*_0 g$. By Corollary \ref{corollary:HWB-lifted-local-coordinate-chart}, after possibly shrinking $\ucal$, there is a coordinate chart $\Phi:\ucal \xrightarrow{\sim} U$ satisfying $\Phi(\zcal_0) = 0$, and a smooth map $f:U \times L \rightarrow X$ such that $f_t:= f(t,\cdot)$ is a free Lagrangian immersion representing the immersed Lagrangian $\Phi^{-1}(t)$, such that
    \[f_0^* \left( i_{\partial_i f} \omega \right) = \theta_i  \qquad \forall\ 1\le i\le m.\]
    Since $f(0,\cdot) = f_0$, the smooth lifting $f(se_i,\cdot)$ of special Lagrangian immersions can be used to compute $R_{f_0}(\Phi^{-1} se_i)$ and $S_{f_0}(\Phi^{-1} se_i)$. By Corollary \ref{corollary:HWB-lifted-local-coordinate-chart}(iii),
    \begin{align*}
        \lim_{h\rightarrow 0} \frac{R_{f_0}(\Phi^{-1}(he_i)) - R_{f_0}}{h} =& \lim_{h\rightarrow 0} \frac{1}{h} \int_{0}^h f_{se_i}^* \left( i_{\frac{df_{se_i}}{ds}} \omega \right) ds \\
        =&\  f_{0}^* \left( i_{\frac{df(se_i, \cdot)}{ds} |_{s=0}} \omega \right) \\
        =&\ f_{0}^* \left( i_{\partial_if(0,\cdot))} \omega \right)\\
        =&\ \theta_i,
    \end{align*}
    \begin{equation}
        \label{eq:HWB-der-of-rel-flux}
        \implies \frac{R_{f_0}}{\partial t_i}(\zcal_0) = \frac{R_{f_0}}{\partial t_i}(\Phi^{-1} 0) = \theta_i \qquad \forall\ 1\le i\le m.
    \end{equation}
    Similarly by Proposition~\ref{prop:HWB-tangentspacetospeciallagrangianswithboundary-1}, if $\star_0$ corresponds to the pullback metric $f_0^*g$,
    \[ f_{0}^* \left( i_{\frac{df(se_i, \cdot)}{ds} |_{s=0}}  \Im (\Omega) \right) = \star_0 f_{0}^* \left( i_{\frac{df(se_i, \cdot)}{ds} |_{s=0}} \omega \right) \]
    \begin{align*}
        \implies \qquad \lim_{h\rightarrow 0} \frac{S_{f_0}(he_i) - S_{f_0}}{h} =& \lim_{h\rightarrow 0} \frac{1}{h} \int_{0}^h f_{se_i}^* \left( i_{\frac{df_{se_i}}{ds}} \Im (\Omega) \right) ds \\
        =&\  f_{0}^* \left( i_{\frac{df(se_i, \cdot)}{ds} |_{s=0}}  \Im (\Omega) \right) \\
        =&\ \star_0 f_{0}^* \left( i_{\frac{df(se_i, \cdot)}{ds} |_{s=0}} \omega \right)\\
        =&\  \star_0  f_{0}^* \left( i_{\partial_if(0,\cdot))} \omega \right)\\
        =&\ \star_0 \theta_i,
    \end{align*}
     \begin{equation}
        \label{eq:HWB-der-of-dual-flux}
        \implies \frac{S_{f_0}}{\partial t_i}(\zcal_0) = \frac{S_{f_0}}{\partial t_i}(\Phi^{-1} 0) = \star_0 \theta_i \qquad \forall\ 1\le i\le m.
    \end{equation}
    Note that since $\theta_1,\dots, \theta_m$ form a basis for $H^1(L, \partial L; \RR)$, $\star_0 \theta_1,\dots, \star_0 \theta_m$ form a basis for $H^{n-1}(L;\RR)$. Conclude from \eqref{eq:HWB-der-of-rel-flux} and \eqref{eq:HWB-der-of-dual-flux} that the derivatives $dR_{f_0}$ and $dS_{f_0}$ are invertible. By the inverse function theorem, $R_{f_0}$ and $S_{f_0}$ are are diffeomorphisms onto their image after shrinking $\ucal$. 
\end{proof}

\begin{definition}
    \label{def:HWB-pointed-coordinate-chart}
    Fix an immersed special Lagrangian $\zcal_0 \in \scal \lcal(X,L; \Lambda_1,\dots , \Lambda_d)$ and an immersion representing it $f_0: L \rightarrow X$. By Lemma~\ref{lemma:HWB-rel-and-dual-are-local-diffeo}, there exists a simply-connected neighborhood $\ucal$ of $\zcal_0$ such that $R_{f_0}$ and $S_{f_0}$ are diffeomorphisms on $\ucal$. Any such triple $(\ucal, \zcal_0 ,f_0)$ is called a called a \textit{pointed coordinate chart} on $\scal \lcal(X,L; \Lambda_1,\dots , \Lambda_d)$.
\end{definition}

A change in the immersion representing the base point $f_0$ is an affine change of coordinates for the charts constructed in Definition \ref{def:HWB-pointed-coordinate-chart}.

\begin{lemma}
\label{lemma:HWB-change-of-immersion}
    Fix a free immersed special Lagrangian $\zcal_0 \in \scal \lcal(X,L; \Lambda_1,\dots , \Lambda_d)$ and a special Lagrangian immersion $f_0: L \rightarrow X$ representing $\zcal_0$. Fix a diffeomorphism $\psi \in \Diff(L)$, and let $\overline{f}_0 := f_0 \circ \psi$. Let $\ucal \subset \scal \lcal$ be a small simply-connected neighborhood of $\zcal_0$ such that $(\ucal, \zcal_0, f_0)$, $(\ucal, \zcal_0, \overline{f}_0)$ are both pointed coordinate charts (Definition \ref{def:HWB-pointed-coordinate-chart}). Then
    \begin{equation}
        \label{eq:HWB-change-of-immersion}
        \begin{aligned}
        R_{\overline{f}_0} \circ R_{f_0}^{-1} (\theta) =&\ \psi^* \theta \qquad \forall\ \theta \in R_{f_0}(\ucal) \subset H^1(L,\partial L; \RR), \\
        \&\quad S_{\overline{f}_0} \circ S_{f_0}^{-1} (\phi) =&\ \psi^* \phi \qquad   \forall\ \phi \in S_{f_0}(\ucal) \subset H^{n-1}(L ; \RR).
    \end{aligned}
    \end{equation}
\end{lemma}
\begin{proof}
    Fix any point $\zcal_1\in \ucal$ and fix a smooth path of free immersed special Lagrangians $\zcal:[0,1] \rightarrow X$ connecting $\zcal_0$ and $\zcal_1$. Consider a smooth lifting of $\zcal$ (Definition \ref{def:smoothpathoflagrangians}), a smooth function
    \[f:[0,1]\times L \rightarrow X,\]
    such that $f(t,\cdot):L \rightarrow X$ is a special Lagrangian immersion representing $\zcal_t$ and $f(0,\cdot) = f_0$ is the given immersion representing $\zcal_0$. Let $f_t:=f(t,\cdot)$. Consider another lifting of the path $\zcal$
    \[\overline{f}:[0,1]\times L \rightarrow X, \quad \overline{f}(t,x) = f(t,\psi(x)) \qquad \forall\ t\in[0,1], x\in L. \]
    Note that $\overline{f}(t,\cdot):L \rightarrow X$ is also a special Lagrangian immersion representing $\zcal_t$ for all $t\in [0,1]$, since $\psi \in \Diff(L)$. Furthermore $\overline{f}(0,\cdot) = \overline{f}_0$ is the given immersion by assumption. Let $\overline{f}_t:= \overline{f}(t, \cdot)$ and note $\overline{f}_t = f_t\circ \psi$. By Definition \ref{def:HWB-def-of-coordinate-chart}, 
    \[R_{f_0}(\zcal_1) = \int_0^1 f_t^* \left( i_{\frac{df_t}{dt}} \omega \right)\]
    \begin{align*}
        \&\quad R_{\overline{f}_0}(\zcal_1) =&\ \int_0^1 \overline{f}_t^* \left( i_{\frac{d\overline{f}_t}{dt}} \omega \right) dt \\
        =&\ \int_0^1 \psi^* {f}_t^* \left( i_{\frac{d{f}_t}{dt} \circ \psi} \omega \right) dt \\
        =&\ \psi^* R_{{f}_0}(\zcal_1)
    \end{align*}
    \begin{equation}
        \label{eq:HWB-change-of-immersion-1}
        \implies \qquad R_{\overline{f}_0} \circ R_{f_0}^{-1} (\theta) = \psi^* \theta \qquad \forall\ \theta \in R_{f_0}(\ucal)  \subset H^1(L,\partial L; \RR).
    \end{equation}
    Similarly, by Definition \ref{def:HWB-def-of-coordinate-chart}, 
    \[S_{f_0}(\zcal_1) = \int_0^1 f_0^* \left( i_{\frac{df_t}{dt}} \Im(\Omega) \right)\]
    \begin{align*}
        \&\quad R_{\overline{f}_0}(\zcal_1) =&\ \int_0^1 \overline{f}_t^* \left( i_{\frac{d\overline{f}_t}{dt}} \Im(\Omega) \right) dt \\
        =&\ \int_0^1 \psi^* {f}_t^* \left( i_{\frac{d{f}_t}{dt} \circ \psi} \Im(\Omega)  \right) dt \\
        =&\ \psi^* S_{{f}_0}(\zcal_1)
    \end{align*}
    \begin{equation}
        \label{eq:HWB-change-of-immersion-2}
        \implies \qquad S_{\overline{f}_0} \circ S_{f_0}^{-1} (\phi) = \psi^* \phi \qquad \forall\ \phi \in S_{f_0}(\ucal) \subset H^{n-1}(L ; \RR). 
    \end{equation}
    Lemma~\ref{lemma:HWB-change-of-immersion} follows from \eqref{eq:HWB-change-of-immersion-1} and \eqref{eq:HWB-change-of-immersion-2}.
\end{proof}

The typical transition map is a change of base point $(\zcal_0,f_0)$ within a pointed coordinate chart $(\ucal, \zcal_0, f_0)$. The corresponding transition maps are affine map and have a very simple form.

\begin{lemma}
\label{lemma:HWB-change-of-basepoint}
    Fix a simply connected open neighborhood $\ucal \subset \scal\lcal(X,L; \Lambda_1,\dots, \Lambda_d)$, two immersed special Lagrangians $\zcal_0, \overline{\zcal}_0 \in \ucal$, and two special Lagrangian immersions representing them respectively $f_0, \overline{f}_0: L \rightarrow X$. Assume that $(\ucal, \zcal_0, f_0)$ and $(\ucal, \overline{\zcal}_0, \overline{f}_0)$ are both pointed coordinate charts (Definition \ref{def:HWB-pointed-coordinate-chart}). Then there exists $\psi \in \Diff(L)$ such that for all $\zcal_1 \in \ucal$,
    \[\begin{aligned}
        R_{\overline{f}_0}(\zcal_1) \circ R_{f_0}^{-1}(\theta) =&\ (\psi^{-1})^*\left( \theta - R_{f_0}(\overline{\zcal}_0) \right) \qquad \forall\ \theta \in R_{f_0}(\ucal) \subset H^1(L, \partial L; \RR) \\
        S_{\overline{f}_0}(\zcal_1) \circ S_{f_0}^{-1}(\phi) =&\ (\psi^{-1})^*\left( \phi- S_{f_0}(\overline{\zcal}_0) \right) \qquad \forall\ \phi \in R_{f_0}(\ucal) \subset H^{n-1}(L; \RR).
    \end{aligned}\]
    The homotopy class of $\psi$, and hence the induced map on cohomology $\psi^*$, depends only on $\ucal$, $f_0$ and $\overline{f}_0$.
\end{lemma}

The homotopy class of $\psi \in \Diff(L)$ in the statement of Lemma~\ref{lemma:HWB-change-of-basepoint} is first identified by Proposition~\ref{prop:HWB-homotopy-class-independent}.

\begin{proposition}
\label{prop:HWB-homotopy-class-independent}
    Fix a simply connected open set $\ucal \subset \scal\lcal(X,L; \Lambda_1,\dots, \Lambda_d)$, two immersed special Lagrangians $\zcal_0, \overline{\zcal}_0 \in \ucal$, and two special Lagrangian immersions representing them respectively $f_0, \overline{f}_0: L \rightarrow X$.  Fix any smooth function $f:[0,1]\times L \rightarrow X$ such that $f(t,\cdot):L \rightarrow X$ is a path of free special Lagrangian immersion representing an immersed Lagrangian in $\ucal$ connecting $\zcal_0$ and $\overline{\zcal}_0$ (Figure~\ref{fig:HWB-change-of-basepoint}). For some $\psi \in \Diff(L)$, assume that
    \[f(0,\cdot) = f_0, \qquad \&\quad f(1,\cdot) = \overline{f}_0 \circ \psi. \]
    Then the homotopy class of $\psi$ is independent of the chosen path $f$, and depends only on the open set $\ucal$, and the end points $f_0$ and $\overline{f}_0$ of $f$. 
\end{proposition}
\begin{proof}
    Let $\widetilde{f}:[0,1]\times L \rightarrow X$ be another such smooth path connecting $\zcal_0$ and $\overline{\zcal}_0$, satisfying
    \[[\widetilde{f}(t,\cdot)] \in \ucal, \qquad  \widetilde{f}(0,\cdot) = f_0, \qquad \&\quad \widetilde{f}(1,\cdot) = \overline{f}_0 \circ \widetilde{\psi}, \]
    for some $\overline{\psi} \in \Diff(L)$ (Definitions~\ref{def:immersionsandimmersedforms}). Since $\ucal$ is simply connected the paths, $[f_t]$ and $[\widetilde{f}_t]$ in $\ucal$ are homotopic with fixed end-points, by a smooth homotopy $\hcal$ satisfying 
    \begin{equation*}
    \begin{aligned}
    \hcal(0,t) = [f(t,\cdot)] \in \ucal, \qquad \qquad \quad &\hcal(1,t)  = [\widetilde{f}(t,\cdot)], \\
    \quad \hcal(u,0) = [f_0] = \zcal_0, \qquad \& \quad &\hcal(u,1) = [f(1,\cdot)] = [\widetilde{f}(1,\cdot)] = \overline{\zcal}_0, \qquad \forall\ t,u\in [0,1].
    \end{aligned}
    \end{equation*}
    By Claim~\ref{claim:HWB-rel-flux-claim1}, there exists a smooth lifting $F:[0,1]^2 \times L \rightarrow X$ of the homotopy $\hcal$, such that $F_{u,t}:= F(u,t,\cdot): L \rightarrow X$ is a free special Lagrangian immersion, satisfying
    \begin{equation*}
        F_{u,0}  = f_0, \qquad  F_{0,t} = f(t,\cdot),\qquad F_{1,t} = \widetilde{f}(t,\cdot), \qquad [F_{u,t}]  =\hcal_{u,t} \in \ucal \qquad  \forall\ u,t\in [0,1].
    \end{equation*}
    By assumption $[F_{u,1}] = \hcal_{u,1} = \zcal_1 = \widetilde{\zcal}_1 = \overline{\zcal}_0$ and so $F_{u,1}$ and $\overline{f}_0$ represent the same immersed special Lagrangian. The exists a smooth family $\psi_u \in \Diff(L)$ such that
    \[ F_{u,1} = \overline{f}_0 \circ \psi_u \qquad \forall\ u\in [0,1].\]
    Since $\overline{f}_0$ is a free immersion,
    \[\overline{f}_0 \circ \psi_0 = F_{0,1} = f_1 = \overline{f}_0 \circ \psi \qquad \& \qquad \overline{f}_0 \circ \psi_1 = F_{1,1} = \widetilde{f}_1 = \overline{f}_0 \circ \widetilde{\psi}, \]
    \[\implies \qquad \psi_0 = \psi \qquad \& \qquad \psi_1 = \widetilde{\psi}.\]
    But $\psi_u$ is a homotopy between $\psi$ and $\widetilde{\psi}$. Since the paths $f$ and $\widetilde{f}$ were arbitrary, conclude that the homotopy class of $\psi$ is independent of the chosen path $f$ and lifting and depends only on the $\ucal$ and the endpoints $f_0, \overline{f}_0$ of $f$.
\end{proof}

\begin{figure}[thbp]
    \centering
    \begin{tikzpicture}
        \draw[thick] (0,-0.1) circle (2.2);
        \filldraw (-1,-1) circle (0.03) node[below] {$\zcal_0$};
        \filldraw (1,-1) circle (0.03) node[below] {$\overline{\zcal}_0$};
        \filldraw (0,1) circle (0.03) node[above] {$\zcal_2$};
        \draw[thick] plot [smooth] coordinates {(-1,-1) (-0.5,-0.95) (0,-1.05) (0.5,-0.95) (1,-1)};
        \draw[thick] plot [smooth] coordinates {(1,-1) (0.7,-0.55) (0.55,-0.05) (0.2,0.55) (0,1)};
        \draw (0,-1.3) node {${f}_t$};
        \draw (0.8,0) node {$\widehat{f}_t$};
        \draw (2.5,0) node {$\ucal$};
    \end{tikzpicture}
    \caption{Change of base point from $(\zcal_0,f_0)$ to $(\overline{\zcal}_0 , \overline{f}_0)$.}
    \label{fig:HWB-change-of-basepoint}
\end{figure}

\begin{proof}[Proof of Lemma~\ref{lemma:HWB-change-of-basepoint}]
Fix a smooth path of special Lagrangian immersions in $\ucal$ connecting $\zcal_0$ and $\overline{\zcal}_0$, say ${\zcal}:[0,1]\rightarrow \ucal$ with $\zcal_1 := \zcal(1) = \overline{\zcal}_0$. Consider a smooth lifting of the path ${\zcal}$, say ${f}:[0,1]\times L \rightarrow X$ satisfying  (Definitions~\ref{def:immersionsandimmersedforms} and \ref{def:immersedlagrangiansandspeciallagrangians})
\begin{equation}
    \label{eq:HWB-change-of-basepoint-f-assumption}
     f(0,\cdot) = f_0, \qquad {f}_t:= {f}(t,\cdot), \ [\overline{f}_t] = \zcal_t \in \ucal \quad \forall\ t\in [0,1], \qquad {f}_1 = \overline{f}_0 \circ \psi 
\end{equation}
for some $\psi \in \Diff(L)$ depending on the choice of path ${\zcal}$ as well as the lifting ${f}$.

Consider a generic element $\zcal_2 \in \ucal$. Fix a path of special Lagrangian immersions connecting $\overline{\zcal}_0$ to $\zcal_2$ in $\ucal$ (Figure~\ref{fig:HWB-change-of-basepoint}), a smooth function $\widehat{f}:[0,1]\times L \rightarrow X$ satisfying
\begin{equation}
    \label{eq:HWB-change-of-basepoint-f-tilde-assumption}
     \widehat{f}_t:= \widehat{f}(t,\cdot),\quad \widehat{f}_0 = \overline{f}_0 \circ \psi, \quad [\widehat{f}_t] \in \ucal \quad \forall\ t\in [0,1], \quad [\widehat{f}_1] = \zcal_2, 
\end{equation}
where $\psi\in \Diff(L)$ is as in \eqref{eq:HWB-change-of-basepoint-f-assumption}. The choice $\widehat{f}_0 = f_1 \circ \psi$ ensures that the paths ${f}_t$ and $\widehat{f}_t$ can be continuously concatenated. Compute $R_{\overline{f}_0  \circ \psi }(\zcal_2)$ and $S_{\overline{f}_0 \circ \psi }(\zcal_2)$ using the path $\widehat{f}_t$ joining $\overline{\zcal}_0$ and $\zcal_2$. By Definition~\ref{def:HWB-def-of-coordinate-chart}
\[ \begin{aligned}
    R_{\overline{f}_0\circ \psi }(\zcal_2) = R_{\widehat{f}_0} =&\ \int_0^1 \widehat{f}_t^* \left( i_{\partial_t \widehat{f}_t} \omega \right) dt \\
    \&\qquad S_{\overline{f}_0\circ \psi }(\zcal_2) = S_{\widehat{f}_0} =&\ \int_0^1 \widehat{f}_t^* \left( i_{\partial_t \widehat{f}_t} \Im(\Omega) \right) dt.
\end{aligned} \]
Similarly compute $R_{f_0}(\zcal_2)$ and $S_{f_0}(\zcal_2)$ using the concatenation of the paths ${f}_t$ and $\widehat{f}_t$ (Figure~\ref{fig:HWB-change-of-basepoint}). By Definition \ref{def:HWB-def-of-coordinate-chart}, Lemma~\ref{lemma:HWB-change-of-immersion}, \eqref{eq:HWB-change-of-basepoint-f-assumption} and \eqref{eq:HWB-change-of-basepoint-f-tilde-assumption},
\[ \begin{aligned}
    R_{f_0}(\zcal_2) =&\ \int_0^1 {f}_t^* \left( i_{\partial_t {f}_t} \omega \right) dt + \int_0^1 \widehat{f}_t^* \left( i_{\partial_t \widehat{f}_t} \omega \right) dt \\
    =&\ R_{f_0}(\zcal_1) + R_{\overline{f}_0\circ \psi} (\zcal_2) \\
    =&\ R_{f_0}(\overline{\zcal}_0) + \psi^* R_{\overline{f}_0}(\zcal_2)
\end{aligned} \]
\begin{equation}
\label{eq:HWB-change-of-basepoint-1}
    \implies \quad R_{\overline{f}_0}(\zcal_2) \circ R_{f_0}^{-1}(\theta) = (\psi^{-1})^*\left( \theta - R_{f_0}(\overline{\zcal}_0) \right) \quad \forall\ \theta \in R_{f_0}(\ucal) \subset H^1(L, \partial L; \RR),
\end{equation}

\[ \begin{aligned}
   \&\qquad  S_{f_0}(\zcal_2) =&\ \int_0^1 {f}_t^* \left( i_{\partial_t {f}_t} \Im(\Omega) \right) dt + \int_0^1 \widehat{f}_t^* \left( i_{\partial_t \widehat{f}_t} \Im(\Omega) \right) dt \\
    =&\ S_{f_0}(\zcal_1) + S_{\overline{f}_0\circ \psi} (\zcal_2) \\
    =&\ S_{f_0}(\overline{\zcal}_0) + \psi^* S_{\overline{f}_0}(\zcal_2)
\end{aligned} \]
\begin{equation}
\label{eq:HWB-change-of-basepoint-2}
    \implies \quad S_{\overline{f}_0}(\zcal_2) \circ S_{f_0}^{-1}(\phi) = (\psi^{-1})^*\left( \phi - S_{f_0}(\overline{\zcal}_0) \right) \quad \forall\ \phi \in S_{f_0}(\ucal) \subset H^{n-1}(L; \RR).
\end{equation}

Lemma~\ref{lemma:HWB-change-of-basepoint} follows from Claim~\ref{prop:HWB-homotopy-class-independent}, \eqref{eq:HWB-change-of-basepoint-1} and \eqref{eq:HWB-change-of-basepoint-2}. 
\end{proof}

A special case of Lemma~\ref{lemma:HWB-change-of-basepoint} is when the two given special Lagrangian immersions $f_0,\overline{f}_0$ are connected by a path of special Lagrangian immersions. In this case the corresponding transition maps are simply translations.

\begin{corollary}
\label{corollary:HWB-change-of-basepoint-special-case}
    Let $\ucal \subset \scal \lcal(X,L; \Lambda_1,\dots , \Lambda_d)$ be simply-connected and fix $\zcal_0, \zcal_1 \in \ucal$. Let $f:[0,1]\times L \rightarrow X$ be a smooth map such that $f_t := f(t,\cdot)$ is a free immersed special Lagrangian with boundary conditions $ \Lambda_1,\dots , \Lambda_d$ such that
    \[[f_0] = \zcal_0,\quad [f_1] = \zcal_1, \quad \&\quad [f_t] \in \ucal.\]
    Then
    \[\begin{aligned}
        R_{f_0}(\zcal_2) =&\ R_{f_1}(\zcal_1) + R_{f_0}(\zcal_1) \\
        S_{f_0}(\zcal_2) =&\ S_{f_1}(\zcal_1) + S_{f_0}(\zcal_1)
    \end{aligned} \qquad \forall\ \zcal_2 \in \ucal.\]
\end{corollary}
\begin{proof}
    Taking $ \overline{\zcal}_0  =\zcal_1$ and $\overline{f}_0 = f_1$ in Lemma~\ref{lemma:HWB-change-of-basepoint}, it follows that $\psi = 0$. Then Corollary~\ref{corollary:HWB-change-of-basepoint-special-case} follows immediately from Lemma~\ref{lemma:HWB-change-of-basepoint}.
\end{proof}

Lemma~\ref{lemma:HWB-change-of-basepoint} deals with the transition maps between two coordinate charts $\ucal_0,\ucal_1$ when each charts contains both base points $(\zcal_0, f_0)$ and $(\zcal_1, f_1)$. In general, the base point of the second coordinate chart might lie outside the first one. In this case the strategy is to apply a change of base point twice, first shifting from $\zcal_0$ to a common base point $\zcal_2 \in \ucal_0 \cap \ucal_1$, and then shifting and then shifting from $\zcal_2$ to $\zcal_1$ (Figure~\ref{fig:HWB-change-of-chart}). This change of base point a priori depends on an arbitrarily chosen point $\zcal_2 \in \ucal_0 \cap \ucal_1$. Fortunately it can be shown that the transition maps are independent of this choice.

\begin{lemma}
\label{lemma:HWB-change-of-chart}
    Fix immersed special Lagrangians $\zcal_0, \zcal_1 \in \scal \lcal(X,L; \Lambda_1,\dots , \Lambda_d)$ and respective immersions representing them $f_0, f_1: L \rightarrow X$. Let $\ucal_0, \ucal_1 \subset \scal \lcal$ respectively be simply-connected neighborhoods of $\zcal_0, \zcal_1$ such that $(\ucal_0, \zcal_0, f_0)$ and $(\ucal_1, \zcal_1, f_1)$ is a pair of pointed coordinate charts. Assume that $\ucal_0 \cap \ucal_1$ is connected. Then 
    \begin{equation}
        \label{eq:HWB-change-of-chart-main}
        \begin{aligned}
         R_{f_1}\circ R_{f_0}^{-1}(\theta) = (\psi_0 \circ \psi_1^{-1})^* \theta - (\psi_0\circ \psi_1)^* R_{f_0}(\zcal_2) + \psi_1^*R_{f_1}(\zcal_2)& \\
          \qquad \forall\ \theta \in H^1(L, \partial L; \RR),& \\
          S_{f_1}\circ S_{f_0}^{-1}(\phi) = (\psi_0 \circ \psi_1^{-1})^* \phi - (\psi_0\circ \psi_1)^* S_{f_0}(\zcal_2) + \psi_1^*S_{f_1}(\zcal_2) & \\
          \qquad  \forall\ \phi \in H^{n-1}(L; \RR).&
    \end{aligned} 
    \end{equation}
    for some $\psi_0, \psi_1 \in \Diff(L)$ and $\zcal_2 \in \ucal_0, \ucal_1$. Furthermore the homotopy class of $\psi_0 \circ \psi_1^{-1}$ depends only on $\ucal_0 $, $ \ucal_1$, $f_0 $ and $ f_1$, and the cohomology classes
    \[ (\psi_0\circ \psi_1)^* R_{f_0}(\zcal_2) - \psi_1^*R_{f_1}(\zcal_2) \quad \& \quad (\psi_0\circ \psi_1)^* S_{f_0}(\zcal_2) - \psi_1^*S_{f_1}(\zcal_2) \]
    are independent of $\zcal_2$ and depend only on $\ucal_0, \ucal_1$, $f_0$ and $f_1$.
\end{lemma}
\begin{proof}
    \begin{figure}[thbp]
    \centering
    \begin{tikzpicture}
        \draw (0.7,0) circle (2.5);
        \draw (2.3,0) circle (2.5);
        \filldraw (-0.5,0) circle (0.03) node[below] {$\zcal_0$};
        \filldraw (3.5,0) circle (0.03) node[below] {$\zcal_1$};
        \filldraw (1.5,-0.5) circle (0.03) node[below] {$\zcal_2$};
        \filldraw (1.5,1) circle (0.03) node[above] {$\zcal_3$};
        \draw[thick] plot [smooth] coordinates {(-0.5,0) (0.5,-0.1) (1,-0.4) (1.5,-0.5)}; 
        \draw[thick] plot [smooth] coordinates {(3.5,0) (2.5,-0.1) (2,-0.4) (1.5,-0.5)}; 
        \draw[thick] plot [smooth] coordinates {(1.5,1) (1.4,0.5) (1.6,0) (1.5,-0.5)}; 
        \draw[thick] plot [smooth] coordinates {(-0.5,0) (0.5,0.3) (1,0.7) (1.5,1)}; 
        \draw[thick] plot [smooth] coordinates {(3.5,0) (2.5,0.3) (2,0.7) (1.5,1)}; 
        \draw (0,2.6) node {$\ucal_0$};
        \draw (3,2.6) node {$\ucal_1$};
        \draw[very thick, -stealth] (0.75,-0.23) -- (0.76,-0.235) node[below] {$\overline{f}_t$};
        \draw [very thick, stealth-] (2.25,-0.235) -- (2.26,-0.23) node[below] {$\widetilde{f}_t$};
        \draw[very thick, -stealth] (0.75,0.46) -- (0.76,0.47) node[above] {$\overline{h}_t$};
        \draw [very thick, stealth-] (2.25,0.47) -- (2.26,0.46) node[above] {$\widetilde{h}_t$};
    \end{tikzpicture}
    \caption{Change of base point from $(\zcal_0,f_0)$ to $(\zcal_2,f_2)$ on $\ucal_0$, and from $(\zcal_1,f_1)$ to $(\zcal_2,f_2)$ on $\ucal_1$.}
    \label{fig:HWB-change-of-chart}
    \end{figure}
    Pick $\zcal_2 \in \ucal_0 \cap \ucal_1$ and a special Lagrangian immersion $f_2:L \rightarrow X$ representing it. The strategy is to change base point from $(\zcal_0,f_0)$ to $(\zcal_2,f_2)$ on $\ucal_0$, and from $(\zcal_1,f_1)$ to $(\zcal_2,f_2)$ on $\ucal_1$ and then restrict both charts to $\ucal_0 \cap \ucal_1$. Since $\ucal_0\cap \ucal_1$ is connected, a generic element $\zcal_3 \in \ucal_0 \cap \ucal_1$ can be connected to $\zcal_2$ via a smooth path lying inside $\ucal_0 \cap \ucal_1$ (Figure~\ref{fig:HWB-change-of-chart}). By Lemma~\ref{lemma:HWB-change-of-basepoint}, there exists $\psi_0, \psi_1 \in \Diff(L)$, with the homotopy classes depending on $\ucal_0, \ucal_1, f_0, f_1, f_2$, such that
    \begin{equation}
        \label{eq:HWB-change-of-chart-eq1}
        \begin{aligned}
        R_{f_2}(\zcal_3) =&\ \psi_0^*\left( R_{f_0}(\zcal_3) - R_{f_0}(\zcal_2) \right) \\
        R_{f_2}(\zcal_3) =&\ \psi_1^*\left( R_{f_1}(\zcal_3) - R_{f_1}(\zcal_2) \right) \\
        S_{f_2}(\zcal_3) =&\ \psi_0^*\left( S_{f_0}(\zcal_3) - S_{f_0}(\zcal_2) \right) \\
        S_{f_2}(\zcal_3) =&\ \psi_1^*\left( S_{f_1}(\zcal_3) - S_{f_1}(\zcal_2) \right)
    \end{aligned} \qquad \forall\ \zcal_3 \in \ucal_0 \cap \ucal_1,
    \end{equation}
    \begin{equation}
        \label{eq:HWB-change-of-chart-eq1.5}
        \implies \quad \begin{aligned}
         R_{f_1}(\zcal_3) =&\ (\psi_1^{-1})^* \psi_0^* R_{f_0}(\zcal_3) - \psi_1^* \psi_0^* R_{f_0}(\zcal_2) + \psi_1^*R_{f_1}(\zcal_2), \\
         S_{f_1}(\zcal_3) =&\ (\psi_1^{-1})^* \psi_0^* S_{f_0}(\zcal_3) - \psi_1^* \psi_0^* S_{f_0}(\zcal_2) + \psi_1^*S_{f_1}(\zcal_2),
    \end{aligned}
    \end{equation}
    \begin{equation}
        \label{eq:HWB-change-of-chart-eq2}
        \implies \quad\begin{aligned}
         R_{f_1}\circ R_{f_0}^{-1}(\theta) =&\ (\psi_0 \circ \psi_1^{-1})^* \theta -(\psi_0\circ \psi_1)^* R_{f_0}(\zcal_2) + \psi_1^*R_{f_1}(\zcal_2), \\
          S_{f_1}\circ S_{f_0}^{-1}(\phi) =&\ (\psi_0 \circ \psi_1^{-1})^* \phi - (\psi_0\circ \psi_1)^* S_{f_0}(\zcal_2) + \psi_1^*S_{f_1}(\zcal_2).
    \end{aligned} 
    \end{equation}
    By \eqref{eq:HWB-change-of-chart-eq2} the transition maps $R_{f_1} \circ R_{f_0}^{-1}$ and $S_{f_1}\circ S_{f_0}^{-1}$ are affine, but a priori both the linear part $(\psi_0 \circ \psi_1^{-1})^*$ and the translation $(\psi_0\circ \psi_1)^* R_{f_0}(\zcal_2) - \psi_1^*R_{f_1}(\zcal_2)$ appear to depend on an arbitrarily chosen point $\zcal_2 \in \ucal_0 \cap \ucal_1$ and lifting $f_2:L \rightarrow X$. Claim \ref{claim:HWB-change-of-chart-claim1} proves $(\psi_0 \circ \psi_1^{-1})^*$ depends only of $f_0$ and $f_1$. Claim \ref{claim:HWB-change-of-chart-claim1.5} proves the cohomology classes $(\psi_0\circ \psi_1)^* R_{f_0}(\zcal_2) - \psi_1^*R_{f_1}(\zcal_2)$ and $(\psi_0\circ \psi_1)^* S_{f_0}(\zcal_2) - \psi_1^*S_{f_1}(\zcal_2)$ are independent of $\zcal_2$ and depend only on $f_0$ and $f_1$.
    
    \begin{claim}
    \label{claim:HWB-change-of-chart-claim1}
        The homotopy class of $\psi_0 \circ \psi_1^{-1}:L \rightarrow L$ depends only on $\ucal_0$, $\ucal_1$, $f_0$ and $f_1$, and in particular is independent of the chosen point $\zcal_2$ and immersion $f_2$.
    \end{claim}
    \begin{proof}
    First make a choice of $\psi_0$ and $\psi_1$ as follows. Consider smooth maps
    \begin{equation}
        \label{eq:HWB-change-of-chart-claim1-eq-0-1}
        \overline{f}, \widetilde{f}:[0,1]\times L \rightarrow X
    \end{equation}
    such that $\overline{f}_t := \overline{f}(t,\cdot)$ is a smooth path of special Lagrangian immersions in $\ucal_0$ from $\zcal_0$ to $\zcal_2$, and $\widetilde{f}_t := \widetilde{f}(t,\cdot)$ is a smooth path of special Lagrangian immersions in $\ucal_1$ from $\zcal_1$ to $\zcal_2$ (Figure~\ref{fig:HWB-change-of-chart}), satisfying
    \begin{equation}
        \label{eq:HWB-change-of-chart-claim1-eq-0-2}
        \overline{f}_0 = f_0 \qquad \& \qquad \widetilde{f}_0 = f_1.
    \end{equation}
        By assumption $\overline{f}_1$, $\widetilde{f}_1$ and $f_2$ all represent the same immersed special Lagrangian $\zcal_2$, and so there exists $\psi_0,\psi_1 \in \Diff(L)$ such that
        \begin{equation}
        \label{eq:HWB-change-of-chart-claim1-eq-0-3}
        \overline{f}_1 = f_2 \circ \psi_0^{-1} \qquad \& \qquad \widetilde{f}_1 = f_2 \circ \psi_1^{-1}. 
        \end{equation}
        By Lemma~\ref{lemma:HWB-change-of-basepoint}, \eqref{eq:HWB-change-of-chart-eq1} holds for this choice of $\psi_0, \psi_1$.
        
        Consider the following path of special Lagrangian immersions
        \begin{equation}
            \label{eq:HWB-change-of-chart-claim1-eq-0-4}
            \widehat{f}:[0,2] \times L \rightarrow X, \qquad  \widehat{f}(t,\cdot) := \begin{cases}
            \overline{f}_t \qquad &t\in [0,1], \\
            \widetilde{f}_{2-t}\circ \psi_1\circ \psi_0^{-1} \qquad &t\in [1,2].
        \end{cases}
        \end{equation}
        $\widehat{f}$ is a concatenation of $\overline{f}$ and the inverse path to $\widetilde{f}_t$, and is well defined since 
        \[\overline{f}_1 = f_2 \circ \psi_0^{-1} = f_2 \circ \psi_1^{-1} \circ \psi_1 \circ \psi_0^{-1} = \widetilde{f}_{2-1} \circ \psi_1 \circ \psi_0^{-1}. \]
        Furthermore $\widehat{f}$ joins the Lagrangians $\zcal_0$ and $\zcal_1$ and
        \begin{equation}
            \label{eq:HWB-change-of-chart-claim1-eq1}
            \&\qquad \widehat{f}_0 = f_0 , \qquad \widehat{f}_2 = f_1 \circ \psi_1 \circ \psi_0^{-1}. 
        \end{equation}
       Apply Proposition~\ref{prop:HWB-homotopy-class-independent} to the immersions $f_1$ and $f_2$, and the path $\widehat{f}_t$. By \eqref{eq:HWB-change-of-chart-claim1-eq1}, the homotopy class of $\psi_1 \circ \psi_0^{-1}$ depends only on $f_0$ and $f_1$.
    \end{proof}

    \begin{claim}
    \label{claim:HWB-change-of-chart-claim1.5}
        The maps $R_{f_0\circ \psi_0 \circ \psi_1^{-1}} - R_{f_1}$ and $S_{f_0\circ \psi_0 \circ \psi_1} - S_{f_1}$ are constant on $\ucal_0 \cap \ucal_1$.
    \end{claim}
    \begin{proof}
        Fix $\zcal_2$, $\overline{f}, \widetilde{f}, \psi_0, \psi_1$ and $\widehat{f}$ as in \eqref{eq:HWB-change-of-chart-claim1-eq-0-1}--\eqref{eq:HWB-change-of-chart-claim1-eq-0-4} (Figure~\ref{fig:HWB-change-of-chart}). Applying Corollary~\ref{corollary:HWB-change-of-basepoint-special-case} to the pair $\widehat{f}_0 = f_0 \circ \psi_0 \circ \psi_1^{-1}$ and $\widehat{f}_1 = f_2 \circ \psi_1^{-1}$ using the path $(\widehat{f}_t)_{t\in [0,1]}$ on $\ucal_0$,
        \begin{equation}
            \label{eq:HWB-change-of-chart-claim1.5-eq1}
            \begin{aligned}
            R_{f_0 \circ \psi_0 \circ \psi_1^{-1}}(\zcal_3) =&\ R_{f_2\circ \psi_1^{-1}}(\zcal_3) + R_{f_0 \circ \psi_0 \circ \psi_1^{-1}}(\zcal_2) \\
            S_{f_0 \circ \psi_0 \circ \psi_1^{-1}}(\zcal_3) =&\ S_{f_2\circ \psi_1^{-1}}(\zcal_3) + S_{f_0 \circ \psi_0 \circ \psi_1^{-1}} (\zcal_2),
        \end{aligned} \qquad \forall\ \zcal_3 \in \ucal_0.
        \end{equation}
        Similarly, applying Corollary~\ref{corollary:HWB-change-of-basepoint-special-case} to the pair $\widehat{f}_2 = f_1$ and $\widehat{f}_1 = f_2 \circ \psi_1^{-1}$ $\ucal_1$ using the path $(\widehat{f}_{2-t})_{t\in [0,1]}$ on $\ucal_1$,
        \begin{equation}
            \label{eq:HWB-change-of-chart-claim1.5-eq2}
        \begin{aligned}
            R_{f_2\circ \psi_1^{-1}}(\zcal_3) =&\ R_{f_1}(\zcal_3) + R_{f_1}(\zcal_2) \\
            S_{f_2\circ \psi_1^{-1}}(\zcal_3) =&\ S_{f_1}(\zcal_3) + S_{f_0 \circ \psi_0 \circ \psi_1^{-1}},
        \end{aligned} \qquad \forall\ \zcal_3 \in \ucal_1.
        \end{equation}
        Combining \eqref{eq:HWB-change-of-chart-claim1.5-eq1} and \eqref{eq:HWB-change-of-chart-claim1.5-eq2}, for all $\forall\ \zcal_3\in \ucal_0 \cap \ucal_1$,
        \begin{align*}
            R_{f_0 \circ \psi_0 \circ \psi_1^{-1}}(\zcal_3) - R_{f_1}(\zcal_3) =&\ R_{f_0 \circ \psi_0 \circ \psi_1^{-1}}(\zcal_2) - R_{f_1}(\zcal_2), \\
             S_{f_0 \circ \psi_0 \circ \psi_1^{-1}}(\zcal_3) - S_{f_1}(\zcal_3) =&\ S_{f_0 \circ \psi_0 \circ \psi_1^{-1}}(\zcal_2) - S_{f_1}(\zcal_2). \tag*{\qedhere}
        \end{align*}
    \end{proof}
    
    \begin{claim}
    \label{claim:HWB-change-of-chart-claim2}
        The cohomology classes $\psi_1^* \psi_0^* R_{f_0}(\zcal_2) - \psi_1^*R_{f_1}(\zcal_2) \in H^1(L, \partial L; \RR)$ and $\psi_1^* \psi_0^* S_{f_0}(\zcal_2) - \psi_1^*S_{f_1}(\zcal_2) \in H^{n-1}(L;\RR)$ depend only on $\ucal_0$, $\ucal_1$, $f_0$ and $f_1$, and in particular is independent of the chosen point $\zcal_2$ and immersion $f_2$.
    \end{claim}
    \begin{proof}
        Fix another point $\zcal_3 \in \ucal_0 \cap \ucal_1$ and free special Lagrangian immersion $f_3: L \rightarrow X$ representing $\zcal_3$. Consider the following construction analogous to \eqref{eq:HWB-change-of-chart-claim1-eq-0-1}--\eqref{eq:HWB-change-of-chart-claim1-eq-0-3}. Fix smooth maps
        \begin{equation}
            \label{eq:HWB-change-of-chart-claim2-eq-0-1}
            \overline{h}, \widetilde{h}:[0,1]\times L \rightarrow X
        \end{equation}
        such that $\overline{h}_t := \overline{h}(t,\cdot)$ is a smooth path of special Lagrangian immersions in $\ucal_0$ from $\zcal_0$ to $\zcal_3$, and $\widetilde{h}_t := \widetilde{h}(t,\cdot)$ is a smooth path of special Lagrangian immersions in $\ucal_1$ from $\zcal_1$ to $\zcal_3$ (Figure~\ref{fig:HWB-change-of-chart}), satisfying
        \begin{equation}
            \label{eq:HWB-change-of-chart-claim2-eq-0-2}
            \overline{h}_0 = f_0 \qquad \& \qquad \widetilde{h}_0 = f_1.
        \end{equation}
        By assumption $\overline{h}_1$, $\widetilde{h}_1$ and $f_3$ all represent the same immersed special Lagrangian $\zcal_3$, and so there exists $\eta_0,\eta_1 \in \Diff(L)$ such that
        \begin{equation}
        \label{eq:HWB-change-of-chart-claim2-eq-0-3}
        \overline{h}_1 = f_3 \circ \eta_0^{-1} \qquad \& \qquad \widetilde{h}_1 = f_3 \circ \eta_1^{-1}. 
        \end{equation}
        By Lemma~\ref{lemma:HWB-change-of-basepoint}
        \begin{equation}
        \label{eq:HWB-change-of-chart-claim2-eq1}
        \begin{aligned}
        R_{f_3}(\zcal_2) =&\ \eta_0^*\left( R_{f_0}(\zcal_2) - R_{f_0}(\zcal_3) \right), \\
        R_{f_3}(\zcal_2) =&\ \eta_1^*\left( R_{f_1}(\zcal_2) - R_{f_1}(\zcal_3) \right), \\
        S_{f_3}(\zcal_2) =&\ \eta_0^*\left( S_{f_0}(\zcal_2) - S_{f_0}(\zcal_3) \right), \\
        S_{f_3}(\zcal_2) =&\ \eta_1^*\left( S_{f_1}(\zcal_2) - S_{f_1}(\zcal_3) \right).
        \end{aligned}
        \end{equation}
        
        \begin{equation}
        \label{eq:HWB-change-of-chart-claim2-eq1.5}
        \implies \quad \begin{aligned}
         R_{f_1}(\zcal_2) =&\ (\eta_1^{-1})^* \eta_0^* R_{f_0}(\zcal_2) - \eta_1^* \eta_0^* R_{f_0}(\zcal_3) + \eta_1^*R_{f_1}(\zcal_3), \\
         S_{f_1}(\zcal_2) =&\ (\eta_1^{-1})^* \eta_0^* S_{f_0}(\zcal_2) - \psi_1^* \eta_0^* S_{f_0}(\zcal_3) + \eta_1^*S_{f_1}(\zcal_3).
        \end{aligned}
        \end{equation}
        \eqref{eq:HWB-change-of-chart-claim2-eq1} and \eqref{eq:HWB-change-of-chart-claim2-eq1.5} are analogous to \eqref{eq:HWB-change-of-chart-eq1} and \eqref{eq:HWB-change-of-chart-eq1.5}. But $(\psi_1^{-1})^* \psi_0^*$ and $(\eta_1^{-1})^* \eta_0^*$ are equal as maps on the cohomology of $L$ because $\psi_0\circ \psi_1^{-1}$ and $\eta_0\circ \eta_1^{-1}$ have the same homotopy class by Claim \ref{claim:HWB-change-of-chart-claim1}. Combining \eqref{eq:HWB-change-of-chart-eq1.5} and \eqref{eq:HWB-change-of-chart-claim2-eq1.5} with Lemma~\ref{lemma:HWB-change-of-immersion} and Claim~\ref{claim:HWB-change-of-chart-claim1.5},
        \begin{align*}
            \psi_1^* \psi_0^* R_{f_0}(\zcal_2) - \psi_1^*R_{f_1}(\zcal_2) =&\  (\psi_1^{-1})^* \psi_0^* R_{f_0}(\zcal_3) -  R_{f_1}(\zcal_3)  \\
            =&\ R_{f_0 \circ \psi_0 \circ \psi_1^{-1}}(\zcal_3) - R_{f_1}(\zcal_3) \\
            =&\ R_{f_0 \circ \psi_0 \circ \psi_1^{-1}}(\zcal_2) - R_{f_1}(\zcal_2) \\
            =&\ (\psi_0\circ \psi_1^{-1})^* R_{f_0}(\zcal_2) -  R_{f_1}(\zcal_2) \\
            =&\ (\eta_0\circ \eta_1^{-1})^* R_{f_0}(\zcal_2) -  R_{f_1}(\zcal_2) \\
            =&\ (\eta_1^{-1})^* \eta_0^* R_{f_0}(\zcal_2) -  R_{f_1}(\zcal_2),   \\
            =&\ \eta_1^* \eta_0^* R_{f_0}(\zcal_3) - \eta_1^*R_{f_1}(\zcal_3)
        \end{align*}
        \begin{align*}
            \&\quad \psi_1^* \psi_0^* S_{f_0}(\zcal_2) - \psi_1^*S_{f_1}(\zcal_2) =&\  (\psi_1^{-1})^* \psi_0^* S_{f_0}(\zcal_3) -  S_{f_1}(\zcal_3)  \\
            =&\ S_{f_0 \circ \psi_0 \circ \psi_1^{-1}}(\zcal_3) - S_{f_1}(\zcal_3) \\
            =&\ S_{f_0 \circ \psi_0 \circ \psi_1^{-1}}(\zcal_2) - S_{f_1}(\zcal_2) \\
            =&\ (\psi_0\circ \psi_1^{-1})^* S_{f_0}(\zcal_2) -  S_{f_1}(\zcal_2) \\
            =&\ (\eta_0\circ \eta_1^{-1})^* S_{f_0}(\zcal_2) -  S_{f_1}(\zcal_2) \\
            =&\ (\eta_1^{-1})^* \eta_0^* S_{f_0}(\zcal_2) -  S_{f_1}(\zcal_2),   \\
            =&\ \eta_1^* \eta_0^* S_{f_0}(\zcal_3) - \eta_1^*S_{f_1}(\zcal_3) \tag*{\qedhere}
        \end{align*}
        \end{proof}
    
    Lemma~\ref{lemma:HWB-change-of-chart} follows from \eqref{eq:HWB-change-of-chart-eq2} and Claim~\ref{claim:HWB-change-of-chart-claim1} and Claim \ref{claim:HWB-change-of-chart-claim2}.
\end{proof}

\begin{corollary}
\label{corollary:HWB-special-affine-manifold}
   Under the assumptions of Lemma~\ref{lemma:HWB-change-of-chart}, the transition maps $R_{f_1}\circ R_{f_0}^{-1}$ and $S_{f_1}\circ S_{f_0}^{-1}$ are volume preserving affine maps on $H^1(L, \partial L; \RR)$ and $H^1(L;\RR)$.
\end{corollary}
\begin{proof}
    By Lemma~\ref{lemma:HWB-change-of-chart}, transition maps are a composition of a translation with the linear map $(\psi_0 \circ \psi_1^{-1})^*$ for some $\psi_0,\psi_1 \in \Diff(L)$. Since translations preserve volumes, it is sufficient to show that $(\psi_0 \circ \psi_1^{-1})^*$ is a volume preserving linear map on cohomology. Let $H^*_{\text{free}}(L;\ZZ)$ denote the free part of the integral cohomology $H^*(L;\ZZ)$. Then
    \[\begin{aligned}
        \det \left(\psi_0^*: H^1_{\text{free}}(L, \partial L; \ZZ) \rightarrow H^1_{\text{free}}(L, \partial L; \ZZ) \right) =&\ 1, \\
        \& \qquad \det \left(\psi_0^*: H^{n-1}_{\text{free}}(L; \ZZ) \rightarrow H^{n-1}_{\text{free}}(L; \ZZ) \right) =&\ 1,
    \end{aligned}\]
    because $\psi_0^*$ and $\psi_1^*$ are invertible linear maps on integral lattices. Conclude that
    \[\begin{aligned}
        \det \left(\psi_0^*: H^1(L, \partial L; \RR) \rightarrow H^1(L, \partial L; \RR) \right) =&\ 1, \\
        \& \qquad \det \left(\psi_0^*: H^{n-1}(L; \RR) \rightarrow H^{n-1}(L; \RR) \right) =&\ 1,.
    \end{aligned} \]
    Thus $(\psi_0\circ \psi_1^{-1})^*$ is a volume preserving linear map.
\end{proof}

\subsection{Proof of Theorem~\ref{thm:HWB-affine-structure-on-moduli-space}}
\label{subsec:HWB-proof-of-affine-structures-thm}

\begin{proof}[Proof of Theorem~\ref{thm:HWB-affine-structure-on-moduli-space}]
    By Lemma~\ref{lemma:HWB-rel-and-dual-are-local-diffeo} every point $\zcal_0 \in \scal\lcal(X,L; \Lambda_1,\dots, \Lambda_d)$ admits an open neighborhood $\ucal_{\zcal_0}$ such that for any special Lagrangian immersion $f_0:L \rightarrow X$ representing $\zcal_0$, $R_{f_0}$ and $S_{f_0}$ (Definition \ref{def:HWB-def-of-coordinate-chart}) are diffeomorphisms on the pointed coordinate chart $(U_{\zcal_0}, \zcal_0, f_0)$.
    
    Shrinking each chart $\ucal_{z\cal_0}$ if necessary, WLOG assume that pair wise intersections of these charts are connected. Then by Lemma~\ref{lemma:HWB-change-of-chart} and Corollary~\ref{corollary:HWB-special-affine-manifold}, the transition maps $R_{f_1} \circ R_{f_0}^{-1}$ and $S_{f_1} \circ S_{f_0}^{-1}$ are volume preserving affine maps for any pair of such charts $(\ucal_{\zcal_0}, \zcal_0, f_0)$ and $(\ucal_{\zcal_1}, \zcal_1, f_1)$.
\end{proof}

\subsection{Isometric embedding and Hessian metric}
\label{subsec:HWB-isometric-and-lagrangian-embedding}

This section is concerned with a proof of Proposition~\ref{prop:HWB-isometric-embedding-intro}. The statement is expanded and proved in Proposition~\ref{prop:HWB-isometric-embedding}. Fix a simply connected neighborhood $\ucal \subset \scal\lcal := \scal\lcal (X,L; \Lambda_1,\dots, \Lambda_d)$. Recall $R_{f_0}:\ucal \rightarrow H^1(L, \partial L; \RR) $ and $S_{f_0}: \ucal \rightarrow H^{n-1}(L;\RR)$ (Definition~\ref{def:HWB-def-of-coordinate-chart}). The goal is to study the image of the product $F_{f_0} := (R_{f_0}, S_{f_0})$.

Let $V : = H^1(L, \partial L;\RR)$, so that by the Poincar\'e--Lefschetz duality, $V^* \cong H^{n-1}(L; \RR)$ and the duality pairing on $V\times V^*$ is generated by the cup product, i.e.,
\[V\times V^* = H^1(L, \partial L;\RR) \times H^{n-1}(L; \RR) \rightarrow \RR,\qquad (\theta, \phi) = \int_L \theta \wedge \phi.\]

\begin{definition}
\label{def:geometryofVtimesV*}
A product $V\times V^*$ of a vector space and its dual has certain natural geometric structures. Let $\langle v,\alpha \rangle$ denote the dual pairing for $v\in V, \alpha \in V^*$.
\begin{enumerate}
    \item $V\times V^*$ has a constant coefficient indefinite bilinear symmetric form
    \begin{equation}
        \label{eq:metriconVtimesV*}
        B((v_1,\alpha_1), (v_2,\alpha_2)) := \frac{1}{2}\left( \langle v_1,\alpha_2 \rangle + \langle v_2,\alpha_1 \rangle \right).
    \end{equation}
    \item $V\times V^*$ has a constant coefficient symplectic form
    \begin{equation}
        \label{eq:symplecticformonVtimesV*}
        W((v_1,\alpha_1), (v_2,\alpha_2)) := \langle v_1,\alpha_2 \rangle - \langle v_2,\alpha_1 \rangle.
    \end{equation}
\end{enumerate}
\end{definition}
Fix a pair of Poincar\'e--Lefschetz dual bases
\begin{equation}
    \begin{aligned}
        \alpha_1,\dots, \alpha_m  \in&\ H^1(L,\partial L;\RR),\\
    \beta_1,\dots, \beta_m \in&\ H^{n-1}(L;\RR),
    \end{aligned} \qquad 
    \int_L \alpha_j \wedge \beta_k = \langle [\alpha_j], [\beta_k] \rangle = \delta_{j,k}.
\end{equation}
The bases define local coordinates 
 \[\begin{aligned}
     V\cong \RR^m,&\ \qquad \sum_{j=1}^m u_j \alpha_j \mapsto u = (u_1,\dots, u_m), \\
     V^*\cong \RR^m,&\ \qquad \sum_{j=1}^m v_j \beta_j \mapsto v = (v_1,\dots, v_m).
 \end{aligned}\]
 The pairing in these coordinates is the standard inner product. In these coordinates,
 \begin{equation}
     \label{eq:HWB-metric-sym-form-in-coordinate-VtimesV*}
     \begin{gathered}
         B = \frac{1}{2} \sum_{j=1}^m du_j \otimes dv_j + dv_j \otimes du_j  = \sum_{j=1}^m du_jdv_j \\
         W = \sum_{j=1}^m du_j \wedge dv_j.
     \end{gathered}
 \end{equation}

It turns out that the indefinite metric $B$ restricts to the $L^2$ Riemannian metric (Definition~\ref{def:HWB-L2-iner-product}) on the embedded manifold $(R_{f_0}, S_{f_0})(\ucal) \subset V\times V^*$ and that the symplectic form vanishes identically on it. Proposition~\ref{prop:HWB-isometric-embedding-intro} is restated in more detail as Proposition~\ref{prop:HWB-isometric-embedding}.

\begin{proposition}
\label{prop:HWB-isometric-embedding}
    Fix a simply connected neighborhood $\ucal \subset \scal\lcal$, an immersed special Lagrangian $\zcal_0 \in \ucal$, and an immersion $f_0:L \rightarrow X$ representing $\zcal_0$. Let $d^*$ and $\star_{f_0}$ denotes the co-exterior derivative and Hodge star operator corresponding to the pullback metric $f_0^*g$. The map $F_{f_0} := (R_{f_0}, S_{f_0}): \ucal \rightarrow V\times V^*$ is well defined by Definition~\ref{def:HWB-def-of-coordinate-chart}. By the Hodge decomposition for manifolds with boundary \cite[Chapter 5, Proposition 9.8]{Taylor2011}, 
    \[\begin{aligned}
        V =&\ H^1(L, \partial L;\RR) \cong \hcal^1_D(L) := \{\alpha \in A^1(L): d\alpha = 0,\ d^*\alpha = 0, \alpha|_{\partial L} \equiv 0\}, \\
        V^* =&\ H^{n-1}(L;\RR) \cong \hcal^{n-1}_N(L) := \{\beta\in A^{n-1}(L): d\beta = 0,\ d^*\beta = 0, (\star_{f_0} \beta)|_{\partial L} \equiv 0\}.
    \end{aligned}\]
    Then the following hold
    \begin{enumerate}[label={(\roman*)}]
        \item The linear map $dF_{f_0}|_{\zcal_0}$ has image inside the diagonal subspace
        \[\left\{(\theta, \star_{f_0} \theta): \theta \in  \hcal_D^1(L) \right\} \subset V\times V^*.\]
        \item $F_{f_0}^*B$ is the standard $L^2$ metric on the moduli space of special Lagrangians, and $F_{f_0}^* W \equiv 0$ (Definition~\ref{def:geometryofVtimesV*}).
        \item $F_{f_0}^* B$ is a Hessian metric.
    \end{enumerate}
\end{proposition}

\begin{proof}~
\begin{enumerate}[label={(\roman*)}]
    \item Let $f:(-\epsilon, \epsilon) \times L \rightarrow X$ be a smooth function such that $f_s := f(s,\cdot)$ is a smooth path of special Lagrangian immersions with $f_0$ equal to the given immersion. The tangent to $f_s$ at $s=0$ is an arbitrary tangent vector at $\zcal_0 \in \ucal$. The desired result now follows from \eqref{eq:HWB-app2-flux-along-path-der-formula}.

    \begin{claim}
    \label{claim:HWB-app2-main-claim}
    \begin{equation}
        \label{eq:HWB-app2-flux-along-path-formula}
        \begin{aligned}
        R_{f_0}([f_{u }]) =&\  \int_0^u \left[ f_{z }^* i_{\frac{df_{z }}{dz}} \omega \right] dz, \\
        S_{f_0}([f_{u }]) =&\ \int_0^u \left[ f_{z }^* i_{\frac{df_{z }}{dz}} \Im(\Omega)\right] dz,
    \end{aligned} \qquad \forall u\in (-\epsilon,\epsilon)
    \end{equation}
    Consequently
    \begin{equation}
        \label{eq:HWB-app2-flux-along-path-der-formula}
        \begin{aligned}
            \left. \frac{d }{d u}\right|_{u=0} R_{f_0}([f_u]) =&\ f_0^* \left(i_{\left. \frac{df_{s}}{ds} \right|_{s=0} } \omega \right) \\
            \& \quad \left. \frac{d }{d u}\right|_{u=0} S_{f_0}([f_u]) =&\ \star_{f_0} f_0^* \left(i_{\left. \frac{df_{s}}{ds} \right|_{s=0} } \omega \right).
        \end{aligned}
    \end{equation}
    \[\]
    \end{claim}
    \begin{proof}
        By Definition~\ref{def:HWB-def-of-coordinate-chart}, $R_{f_0}(u ) = RF((f_{s })_{s=0}^u)$ and $S_{f_0}(u ) = SF((f_{s })_{s=0}^u)$ computed along the path of Lagrangians $(f_{s })_{s=0}^u$ starting at $f_0$ and ending at $f_{u }$. Then by \eqref{eq:HWB-intro-rel-flux} and \eqref{eq:HWB-intro-dual-flux},
        \begin{align*}
            R_{f_0}(u ) =&\  RF((f_{s })_{s=0}^u) = \int_0^1  \left[ f_{us }^* i_{\frac{df_{us }}{ds}} \omega  \right] ds \\
            =&\ \int_0^u \left[ f_{z }^* i_{u\frac{df_{z }}{dz}} \omega  \right] \frac{dz}{u} \qquad (\text{with }z = su) \\
            =&\ \int_0^u \left[ f_{z }^* i_{\frac{df_{z }}{dz}} \omega \right] dz. \\
            S_{f_0}(u ) =&\  SF((f_{s })_{s=0}^u) = \int_0^1  \left[ f_{us }^* i_{\frac{df_{us }}{ds}} \Im(\Omega)  \right] ds \\
            =&\ \int_0^u \left[ f_{z }^* i_{u\frac{df_{z }}{dz}} \Im(\Omega)  \right] \frac{dz}{u} \qquad (\text{with }z = su) \\
            =&\ \int_0^u \left[ f_{z }^* i_{\frac{df_{z }}{dz}} \Im(\Omega) \right] dz.
        \end{align*}
    Differentiating \eqref{eq:HWB-app2-flux-along-path-formula} at $u=0$,
    \[ \begin{aligned}
        \left. \frac{d R_{f_0}}{d u}\right|_{u=0} =&\  \left. \frac{d}{du} \int_0^u \left[ f_{z}^* i_{\frac{df_{z}}{dz}} \omega \right] dz \right|_{u=0}  =\left. \left[ f_{u}^* i_{\frac{df_{u}}{du}} \omega\right] \right|_{u=0}  \\
        =&\ \left[ f_{0}^* i_{\left. \frac{df_{u}}{du} \right|_{u=0}} \omega \right].
        \end{aligned}\]
        Similarly differentiating $S_{f_0}$,
        \[ \begin{aligned}
        \left. \frac{d S_{f_0}}{d u}\right|_{u=0} =&\  \left. \frac{d}{du} \int_0^u \left[ f_{z}^* i_{\frac{df_{z}}{dz}} \Im(\Omega) \right] dz \right|_{u=0}  = \left. \left[ f_{u}^* i_{\frac{df_{u}}{du}} \Im(\Omega) \right] \right|_{u=0}  \\
        =&\ \left[ f_{0}^* i_{\left. \frac{df_{u}}{du} \right|_{u=0}} \Im(\Omega) \right].
        \end{aligned}\]
        By Proposition~\ref{prop:HWB-tangentspacetospeciallagrangianswithboundary-1}
        \[ \left. \frac{d S_{f_0}}{d u}\right|_{u=0} = \left[ f_{0}^* i_{\left. \frac{df_{u}}{du} \right|_{u=0}} \Im(\Omega) \right] = \star_{f_0} \left[ f_{0}^* i_{\left. \frac{df_{u}}{du} \right|_{u=0}} \omega \right]. \tag*{\qedhere}\]
    \end{proof}

    \item Fix a pair of harmonic 1-forms vanishing with Dirichlet boundary conditions $\theta_1, \theta_2 \in Z_D^1\cap cZ^{n-1}(L)$ representing linearly independent classes in $H^1(L,\partial L; \RR)$. Compute $F_{f_0}^*B$. By Corollary~\ref{corollary:HWB-lifted-local-coordinate-chart}, there is a two parameter family of special Lagrangian immersions with boundary conditions $\Lambda_1,\dots, \Lambda_d$
    \[ f_t: L \rightarrow X, \qquad t = (t_1,t_2)\in (-\epsilon, \epsilon)^2,\]
    such that $f_{(0,0)} = f_0$ is the given immersion, and for $e_1 = (1,0)$ and $e_2 = (0,1)$,
    \begin{equation}
        \label{eq:HWB-app2-generic-tangents-1}
        \theta_1 = f_0^* \left(\left.  i_{\frac{df_{se_1}}{ds}} \omega  \right|_{s=0} \right) \quad \& \quad \theta_2 = f_0^* \left(\left.  i_{\frac{df_{se_2}}{ds} } \omega  \right|_{s=0}\right) . 
    \end{equation}
    Note that by Proposition~\ref{prop:HWB-tangentspacetospeciallagrangianswithboundary-1}
    \begin{equation}
        \label{eq:HWB-app2-generic-tangents-2}
        \star_{f_0} \theta_1 = f_0^* \left( \left.  i_{\frac{d f_{se_1}}{ds} } \Im(\Omega) \right|_{s=0}  \right)  \quad \& \quad \star_{f_0} \theta_2 = f_0^* \left(\left. i_{ \frac{d f_{se_2}}{ds} } \Im(\Omega) \right|_{s=0}\right).
    \end{equation}
    Compute the quadratic forms $F_{f_0}^*B$ and $F_{f_0}^*W$ along the tangent plane at $f_0$ to the generic surface $[f_t]_{t}$ inside $\scal\lcal$. By Claim~\ref{claim:HWB-app2-main-claim}, \eqref{eq:metriconVtimesV*}, \eqref{eq:HWB-app2-generic-tangents-1} and \eqref{eq:HWB-app2-generic-tangents-2}
    \begin{align}
        & F_{f_0}^*B\left( \frac{\partial }{\partial t_1} [f_t],  \frac{\partial }{\partial t_2} [f_t] \right) \\
        &\qquad = B\left( \frac{\partial }{\partial t_1} \Big(R_{f_0}([f_t]), S_{f_0}([f_t]) \Big),\ \frac{\partial }{\partial t_2} \Big(R_{f_0}([f_t]), S_{f_0}([f_t])\Big) \right) \nonumber \\
        &\qquad = B\left((\theta_1, \star_{f_0} \theta_1), (\theta_2, \star_{f_0} \theta_2)\right) \nonumber \\
        &\qquad = \frac{1}{2} \left( (\theta_1, \star_{f_0} \theta_2)  + (\theta_2, \star_{f_0} \theta_1) \right) \\
        &\qquad =  \frac{1}{2} \left( \int_L \theta_1 \wedge \star_{f_0} \theta_2  + \int_L \theta_2 \wedge \star_{f_0} \theta_1 \right)  \nonumber \\
        &\qquad = \int_L \theta_1 \wedge \star_{f_0} \theta_2 = \int_L f_0^* g(\theta_1, \theta_2) d\vol_{f_0^*g}, \label{eq:HWB-app2-eq3}
    \end{align}
    where the last line follows from the definition of the Hodge star operator. Combining Definition~\ref{def:HWB-L2-iner-product}, \eqref{eq:HWB-app2-generic-tangents-1} and \eqref{eq:HWB-app2-eq3} conclude that $F_{f_0}^*B$ is equivalent to the $L^2$ metric on $\scal\lcal$.
    
    Similarly by  Claim~\ref{claim:HWB-app2-main-claim}, \eqref{eq:symplecticformonVtimesV*}, \eqref{eq:HWB-app2-generic-tangents-1} and \eqref{eq:HWB-app2-generic-tangents-2},
    \begin{align*}
        & F_{f_0}^*W\left( \frac{\partial }{\partial t_1} [f_t],  \frac{\partial }{\partial t_2} [f_t] \right) \\
        &\qquad = W\left( \frac{\partial }{\partial t_1} \Big(R_{f_0}([f_t]), S_{f_0}([f_t]) \Big),\ \frac{\partial }{\partial t_2} \Big(R_{f_0}([f_t]), S_{f_0}([f_t])\Big) \right) \\
        &\qquad =  W\left((\theta_1, \star_{f_0} \theta_1), (\theta_2, \star_{f_0} \theta_2)\right)  \\
        &\qquad =  (\theta_1, \star_{f_0} \theta_2) - (\theta_2, \star_{f_0} \theta_1) \\
        &\qquad =  \int_L \theta_1 \wedge \star_{f_0} \theta_2 - \int_L \theta_2 \wedge \star_{f_0} \theta_1 \\
        &\qquad =  \int_L f_0^* g(\theta_1, \theta_2) d\vol_{f_0^*g} - \int_L f_0^* g(\theta_2, \theta_1) d\vol_{f_0^*g} = 0.
    \end{align*}
    Conclude that $F_{f_0}^*W \equiv 0$.

    \item Use the local coordinates \eqref{eq:HWB-metric-sym-form-in-coordinate-VtimesV*}. In these coordinates, $F_{f_0}$ is represented by $u = R_{f_0}$ and $v = S_{f_0}$. Note that $R_{f_0}$ and $S_{f_0}$ are both invertible on $\ucal$ and so $F_{f_0}^*du \neq 0$ and $F_{f_0}^* dv \neq =0$. By (ii) and \eqref{eq:HWB-metric-sym-form-in-coordinate-VtimesV*}, $F_{f_0}^* W = F_{f_0}^* (du\wedge dv) = 0$ and $F_{f_0}^*du \neq 0$. Thus the image $F_{f_0}(\ucal)$ is a gradient graph over its projection to the first component, i.e., there exists $h:\RR^m \rightarrow \RR$ such that
    \[v = \nabla h(u) \qquad \forall\ u \in R_{f_0}(\ucal).\]
    Then by \eqref{eq:HWB-metric-sym-form-in-coordinate-VtimesV*}, $F_{f_0}^*B$ is a Hessian metric in the affine coordinate chart $u = R_{f_0}$,
    \[ F_{f_0}^*B = F_{f_0}^*\left( \sum_{j=1}^m du_jdv_j \right) = \frac{\partial^2 h}{\partial u_i \partial u_j} du_idu_j. \tag*{\qedhere}\]
    \end{enumerate}
\end{proof}

\section{Generalizations}
\label{sec:HWB-generalizations}

This section discusses generalizations of Proposition~\ref{prop:HWB-tangentspacetospeciallagrangianswithboundary-1}, Proposition~\ref{prop:HWB-tangentspacetospeciallagrangianswithboundary-1}, and Theorem~\ref{thm:HWB-affine-structure-on-moduli-space} beyond Ricci flat Calabi--Yau manifolds $(X,\omega, J ,\Omega)$ from Definition~\ref{def:HWB-calabi-Yau-manifold}.

\begin{definition}
    An \textit{almost Calabi--Yau manifold} is a K\"ahler manifold $(X, \omega, J)$ with a trivial canonical bundle $K_X$. Since $K_X = \Lambda^{(n,0)}T^*X$ is trivial, $X$ admits a a nowhere vanishing $(n,0)$-form $\Omega$. There exists a smooth function $\rho:X \rightarrow (0,\infty)$ such that
    \begin{equation}
        \label{eq:HWB-almost-Calabi-Yau-equation}
        (-1)^{\frac{n(n-1)}{2}} \left( \frac{\sqrt{-1}}{2}\right)^n \Omega \wedge \overline{\Omega} = \rho^2 \frac{1}{n!} \omega^n.
    \end{equation}
    Compare \eqref{eq:HWB-almost-Calabi-Yau-equation} with \eqref{eq:HWB-calabi-Yau-equation}. If $\omega$ is in addition Ricci flat, then $\rho\equiv 1$ and $(X,\omega, J, \Omega)$ would be a Calabi--Yau manifold.
\end{definition}

On an almost Calabi--Yau manifold, special Lagrangians are not minimal surfaces for the Riemannian metric $g$. It turns out that special Lagrangians are minimal surfaces for the conformal metric
\begin{equation}
    \label{eq:HWB-almost-CY-conformal-metric}
    \widetilde{g} = \rho^{\frac{-2}{n}} g.
\end{equation}

In the prequel, it was noted that Proposition~\ref{prop:HWB-tangentspacetospeciallagrangianswithboundary-1} and Theorem~\ref{thm:HWB-tangentspacetospeciallagrangianswithboundary-2} generalize to almost Calabi--Yau manifolds when the pullback metric $f_0^*g$ is replaced by the conformally equivalent on $f_0^* \widetilde{g}$ \cite[{\S}5]{vasanth-hitchin-1}. Similarly Theorem~\ref{thm:HWB-rel-flux-well-defined}, Theorem~\ref{thm:HWB-dual-flux-well-defined} and Theorem~\ref{thm:HWB-affine-structure-on-moduli-space} all generalize to almost Calabi--Yau manifolds in a straight forward manner after replacing the pullback Riemannian metric $f^*g$ by the conformal metric $f^* \widetilde{g}$ defined by \eqref{eq:HWB-almost-CY-conformal-metric}. The proofs follow an identical structure.

\printbibliography

\textsc{University of Maryland, College Park, Maryland}

\href{mailto:pvasanth@umd.edu}{pvasanth@umd.edu}

\end{document}